\documentclass[final,1p,times]{elsarticle}
\usepackage{amsmath,amssymb,amsthm,amscd,appendix}
\usepackage[matrix,arrow,curve]{xy}

\usepackage{graphicx}
\usepackage{comment}
\usepackage{listings}
\usepackage{mathtools, nccmath}

\usepackage{tikz}
\usetikzlibrary{arrows}

\usepackage{changepage}
\usepackage[colorlinks=true,linkcolor=blue]{hyperref}
\usepackage[noabbrev]{cleveref}
\usepackage{overpic}
\usepackage{contour}
\usepackage[normalem]{ulem}

\newtheorem{formula}{}[section]
\newtheorem{proposition}[formula]{Proposition}
\newtheorem{corollary}[formula]{Corollary}
\newtheorem{lemma}[formula]{Lemma}
\crefname{lemma}{Lemma}{Lemmas}
\newtheorem{theorem}[formula]{Theorem}

\newtheorem*{theorem*}{Theorem}

\theoremstyle{definition}
\newtheorem{definition}[formula]{Definition}
\crefname{definition}{Def.}{Definitions}

\theoremstyle{remark}
\newtheorem{remark}[formula]{Remark}
\crefname{remark}{Remark}{Remarks}

\newtheorem*{problem*}{Problem}

\crefname{section}{Section}{sections}
\crefname{figure}{Figure}{figures}
\crefname{equation}{}{equations}
\crefname{appendix}{Appendix}{}

\begin{document}

\title{Flexible Kokotsakis Meshes with Skew Faces: Generalization of the Orthodiagonal Involutive Type}

\author{Alisher Aikyn}
\ead{alisher.aikyn@kaust.edu.sa}
\author{Yang Liu}
\author{Dmitry A.~Lyakhov}
\ead{dmitry.lyakhov@kaust.edu.sa}
\author{Florian Rist}
\ead{florian.rist@kaust.edu.sa}
\author{Helmut Pottmann}
\ead{helmut.pottmann@kaust.edu.sa}
\author{Dominik L.~Michels}
\ead{dominik.michels@kaust.edu.sa}




\address{KAUST Visual Computing Center}

\def\sgn{\mathrm{sgn}\,}
\def\bideg{\mathrm{bideg}\,}
\def\tdeg{\mathrm{tdeg}\,}
\def\sdeg{\mathrm{sdeg}\,}
\def\grad{\mathrm{grad}\,}
\def\ch{\mathrm{ch}\,}
\def\sh{\mathrm{sh}\,}
\def\th{\mathrm{th}\,}
\def\RZ{\mathbb{R}\mathcal{Z}}
\def\mod{\mathrm{mod}\,}
\def\In{\mathrm{In}\,}
\def\Im{\mathrm{Im}\,}
\def\Ker{\mathrm{Ker}\,}
\def\Hom{\mathrm{Hom}\,}
\def\Tor{\mathrm{Tor}\,}
\def\rk{\mathrm{rk}\,}
\def\codim{\mathrm{codim}\,}

\def\ko{{\mathbf k}}
\def\sk{\mathrm{sk}\,}
\def\RC{\mathrm{RC}\,}
\def\gr{\mathrm{gr}\,}

\def\R{{\mathbb R}}
\def\C{{\mathbb C}}
\def\Z{{\mathbb Z}}
\def\A{{\mathcal A}}
\def\B{{\mathcal B}}
\def\K{{\mathcal K}}
\def\M{{\mathcal M}}
\def\N{{\mathcal N}}
\def\E{{\mathcal E}}
\def\G{{\mathcal G}}
\def\D{{\mathcal D}}
\def\F{{\mathcal F}}
\def\L{{\mathcal L}}
\def\V{{\mathcal V}}
\def\H{{\mathcal H}}




\begin{abstract}
In this paper, we introduce and study a remarkable class of mechanisms formed by a $3 \times 3$ arrangement of rigid quadrilateral faces with revolute joints at the common edges. \textcolor{black}{In contrast to the
well-studied Kokotsakis meshes with a quadrangular base, we do not assume the planarity of the quadrilateral faces.  Our mechanisms} are a generalization of Izmestiev's orthodiagonal involutive type of Kokotsakis meshes formed by planar quadrilateral faces. The importance of this Izmestiev class is undisputed as it represents the first known flexible discrete surface -- T-nets -- which has been constructed by Graf and Sauer. Our algebraic approach yields a complete characterization of all \textcolor{black}{flexible $3 \times 3$ quad meshes} of the orthodiagonal involutive type \textcolor{black}{up to some degenerated cases}. It is shown that one has \textcolor{black}{a~maximum of} 8 degrees of freedom to construct such mechanisms. This is illustrated by several examples, including cases which could not be realized using planar faces. We demonstrate the practical realization of the proposed mechanisms by building a physical prototype using stainless steel. In contrast to plastic prototype fabrication, we avoid large tolerances and inherent flexibility.
\end{abstract}
\maketitle

\setcounter{section}{0}
\section{Introduction}

The growing interest in flexible or deployable structures is rooted in their widespread utility, from robotics to solar cells, meta-materials, art, and architecture. \textcolor{black}{In architecture, these structures help to simplify construction processes, open new design possibilities, and increase the utility of buildings. Santiago Calatrava Valls demonstrated this in large-scale projects like the Florida Polytechnic University, the UAE Pavilion at the Expo in Dubai, and the Quadracci Pavilion Milwaukee Art Museum. Norman Foster's Bund Finance Center is another example. Currently, architects resort to a minimal set of simple mechanisms. We unlock a large class of exciting mechanisms for future use in architecture and design. Unlike often investigated kinematic frame structures \cite{escrig1993,calatrava1981,li2020}, the presented quad mechanisms allow complete covering of a surface, which is essential for many applications.}


Deployable structures that have a flat state or even collapse to a straight or curve-like shape are investigated as means to increase the efficiency of construction processes. Even when confining
to rigid components of the structure, there is a wide variety of recent research. We mention rigid-foldable origami \cite{demaine-book,evans2015,evans2016,ciang-2019-cf,song-2017,Tac09a,Tac10a,tachi-2010,tachi-2011-onedof,tachi-2013-composite,tachi2016} and its use to approximate given target surfaces \cite{dang-feng,feng-dang}, or so-called programmable meta-materials, which are based on special patterns and essentially constitute mechanisms, typically with many degrees of freedom \cite{Callens2018,dieleman2020,Dudte2016,auxetic-2022,Konakovic:2016:BDC:2897824.2925944,konakovic2018,silverberg2014}.

The present work is a step towards \emph{transformable design with quad mesh mechanisms} \textcolor{black}{ (see Fig.~\ref{fig:SOM} for an example). These are quad meshes with regular combinatorics, possibly apart from a few isolated combinatorial singularities. Their faces are rigid bodies and all edges where two faces meet are revolute joints. In general, a quad mesh with rigid faces is overall rigid. The meshes we are interested in are special mechanisms that are capable of
performing a one-parameter motion when one quad is fixed (see Fig.~\ref{fig:motion-sequence}). In contrast to almost all prior work, we do not assume the quadrilateral faces to be planar. We also use the term 
\emph{flexible quad mesh} for such a quad mesh mechanism.}

\begin{figure}[t]
\centering
\includegraphics[width=0.3\textwidth]{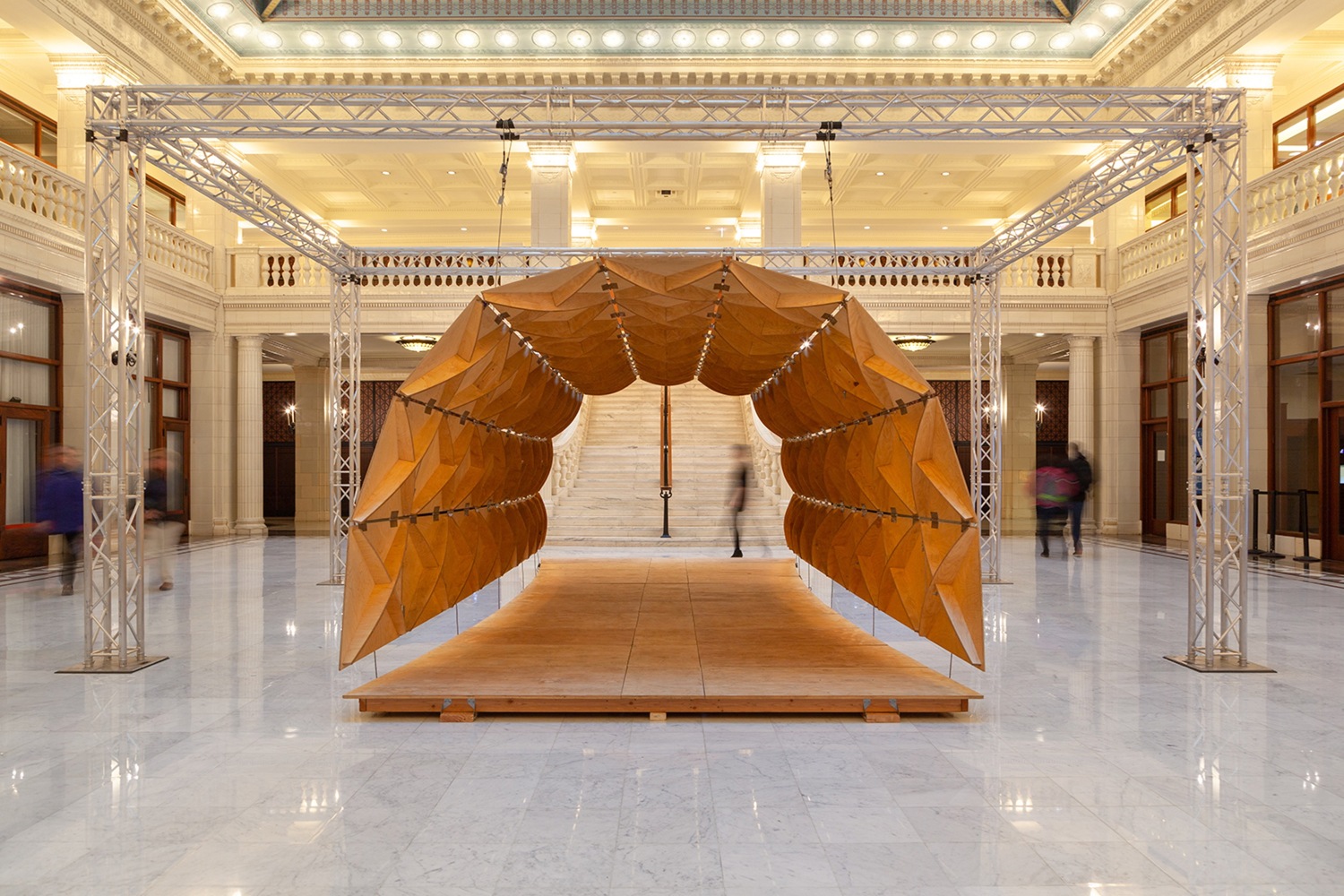}
	\hspace{1cm}
 \includegraphics[width=0.3\textwidth]{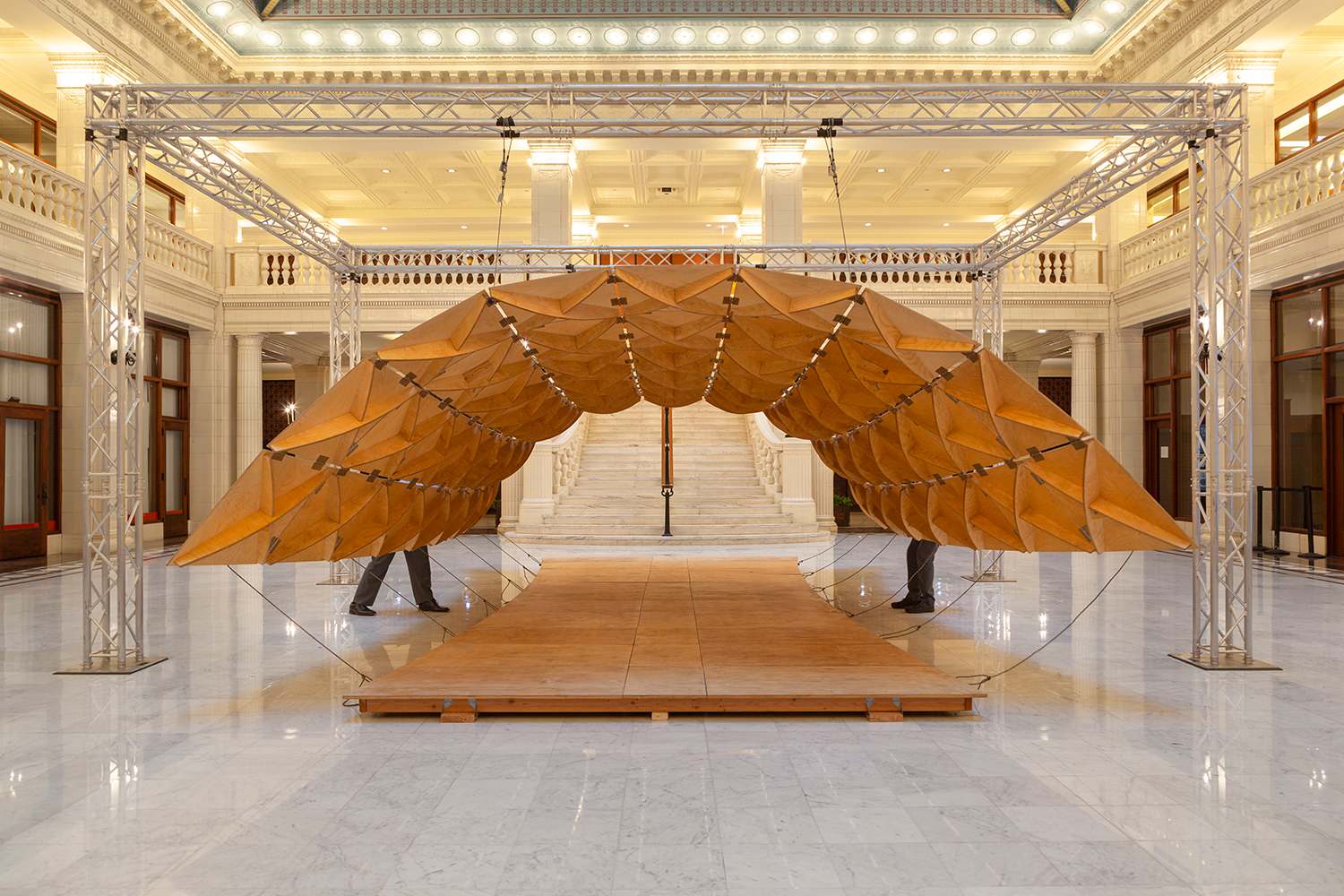}
\caption{\textcolor{black}{Example of transformable design with a quad mesh mechanism: {\it Kinematic Sculpture} by Skidmore, Owings \&\ Merrill,
an installation for the 2018 Chicago Design week \cite{Mitchell2018}.
Here the underlying flexible mesh is a so-called Voss net, a special flexible quad mesh with planar faces.}} \label{fig:SOM}
\end{figure}

\textcolor{black}{While a simply connected quad mesh of $2 \times n$ faces is always
flexible, a $3 \times 3$ quad mesh is in general rigid. This hints already to the fact that the basic components in a flexible quad mesh are the are so-called \emph{elementary cells} or \emph{complexes}, which are the $3 \times 3$ submeshes contained in it. It can be shown with the same proof as in Schief et al.~\cite{Schief2008}  that a quad mesh is flexible if and only if all its $3 \times 3$ submeshes are flexible. Hence, our focus is on flexible $3 \times 3$ meshes. We present a rich class of such flexible $3 \times 3$ quad meshes. Figure~\ref{fig:motion-sequence} shows an example and for its motion refer to a video on YouTube, \cite[1:03-2:07]{kokotsakis23}. }

\begin{figure}
\centering
\includegraphics[width=\textwidth]{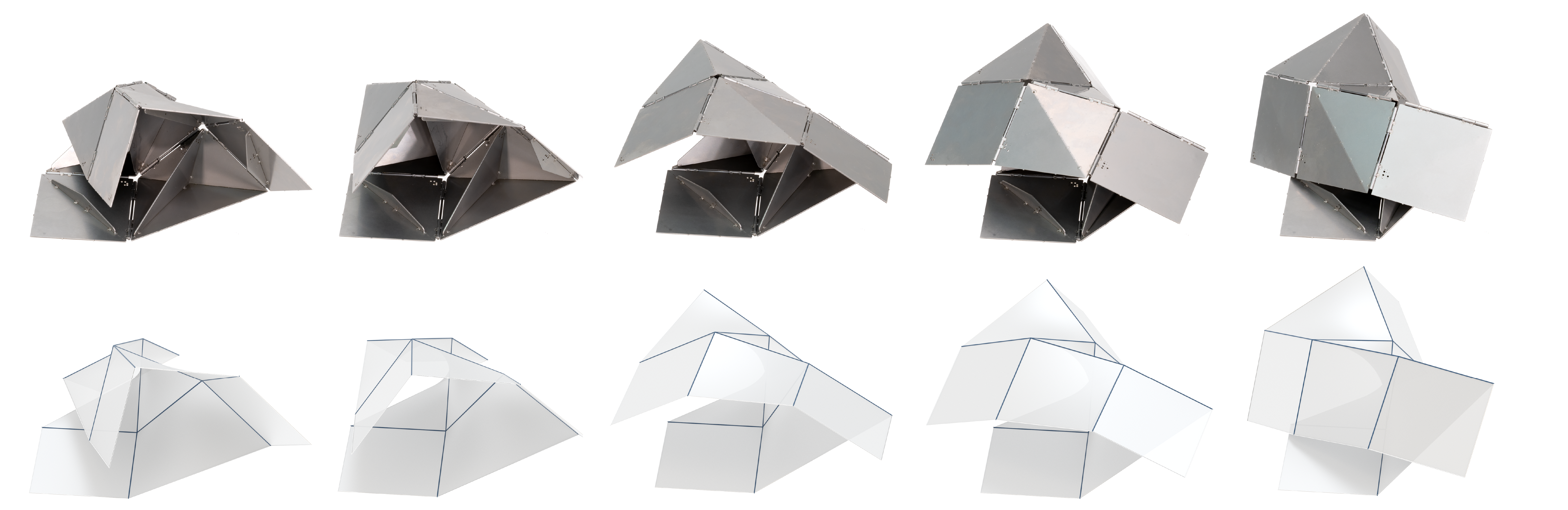}
\vspace{-0.4cm}
\caption{\textcolor{black}{One-parametric motion of a flexible 3x3 quad mesh with non-planar faces computed with our method. The lower left quad is fixed. In the upper row, we show a physical model and in the lower row the corresponding states in a digital model. There, the blue edges are those where a hinge is placed. To avoid confusion caused by diagonals in the non-planar quads, the quad faces in the digital model are filled by bilinear surfaces.}}
\label{fig:motion-sequence}
\end{figure}

\subsection{Prior work}

If one goes beyond quad meshes of $3 \times 3$ faces, only very few solutions are known, especially if they do not possess a flat state as in origami. \textcolor{black}{Most of these} quad meshes have planar faces and are discrete counterparts of well-studied smooth surfaces, namely Voss surfaces (characterized by a conjugate net of geodesics) \cite{voss1888uber} and profile affine surfaces \cite{sauer-graf,sauer:1970} which generalize molding surfaces that are traced out by a planar profile whose plane rolls on a cylinder. The discrete counterparts of Voss surfaces, so-called Voss nets, are reciprocal parallel to discrete models for surfaces of constant negative Gauss curvature  \cite{wunderlich-1951,sauer:1970}. Discrete profile affine surfaces are known as T-nets, as their flat faces are all trapezoids. Both types have recently received interest for applications in transformable design and architecture \cite{voss-IASS,t-nets-IASS,Mitchell2018,archdaily18}.

Flexible polyhedra have a long history in mathematical research. Famous examples include Bricard's flexible octahedra
\cite{bricard1897} and the flexible Kokotsakis meshes \cite{kokotsakis33,stachel2010}. The latter have a planar $n$-gon $P$ in the center, which is
surrounded by a belt of planar faces so that 3 faces meet at each vertex of $P$.  The case with a central quad $P$ 
is equivalent to a flexible $3 \times 3$ mesh of planar quads. \textcolor{black}{A classification of all types of such building blocks has been given by Izmestiev \cite{izmestiev-2017} by studying the complexified configuration space of the polyhedron. The idea behind the classification is briefly described in \cref{sec3}.  }
One remarkable type in Izmestiev's classification is called {\it orthodiagonal anti-involutive} and has been thoroughly studied by Erofeev and Ivanov \cite{EI20}. They have obtained a parametrization of all real examples and described an algorithm for their construction.

\textcolor{black}{Substantial progress on the composition of Izmestiev's flexible $3 \times 3$ meshes to larger flexible quad mesh mechanisms has recently been made by Dieleman et al. \cite{dieleman2020} and for some specific classes, such as linear compounds and conical types, by He and Guest~\cite{he2020rigid}.}

Another stream of relevant research is related to infinitesimal flexibility. Sauer \cite{sauer:1970} studied first order flexibility in great detail, not only for the case of planar quad faces. However, he mentions that there is no known example for finite flexibility in the presence of non-planar faces. Schief et al. \cite{Schief2008} showed that in the smooth limit and for planar faces, second order infinitesimal flexibility implies finite flexibility. This is not true for the discrete setting, but one has a remarkable type of integrable systems \cite{Schief2008}. 

The first example of a flexible mesh with non-planar faces has recently been presented by Nawratil \cite{Nawratil2022}
in a study of generalized Kokotsakis belts with non-planar quads surrounding the central polyline. This is a generalization of the isogonal type in the planar quad case, where in each vertex both pairs of opposite angles are equal (as in Voss nets) or
supplementary (as in flat foldable origami). 

\subsection{Our contribution}
In the present paper, we present a larger novel class of flexible $3 \times 3$ meshes \textcolor{black}{allowing also skew (= non-planar) quad faces}. 
We refer to this generalization of Kokotsakis meshes
with quadrangular bases as \textcolor{black}{\emph{skew Kokotsakis meshes}}. 
We follow the classical approach proposed by Bricard \textcolor{black}{and reformulate the flexibility problem in Euclidean space as one on the sphere}. This means that we can proceed with the polynomial system expressed in new variables, half-tangents of adjacent dihedral angles. Flexibility then means the existence of a one-parameter real-valued solution set of this
system.

We extend the approach of Izmestiev by constructing one remarkable class, \textcolor{black}{namely} orthodiagonal spherical quadrilaterals with involutive couplings. We derive explicit algebraic expressions to construct flexible \textcolor{black}{\emph{skew Kokotsakis meshes}}. These mechanisms have \textcolor{black}{maximum} 8 degrees of freedom, \textcolor{black}{not counting the one degree of freedom in a flexion. This is wider in comparison to the planar case. Our mechanisms} possess special geometric properties which generalize the classical ones of T-nets.

\subsection{Overview}\label{sec-overview}


\begin{figure}[t]
\centering
\includegraphics[width=0.35\textwidth]{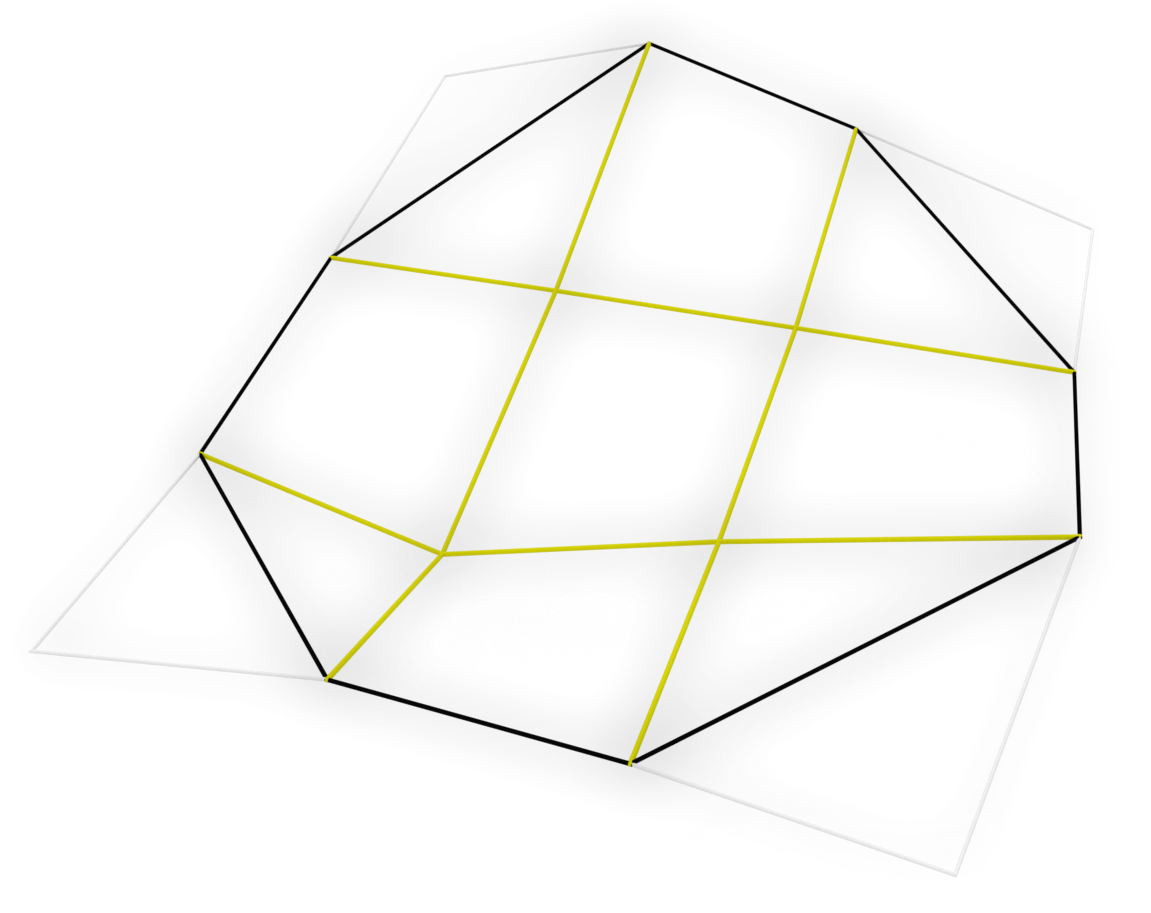}
	\hspace{1cm}
 \includegraphics[width=0.35\textwidth]{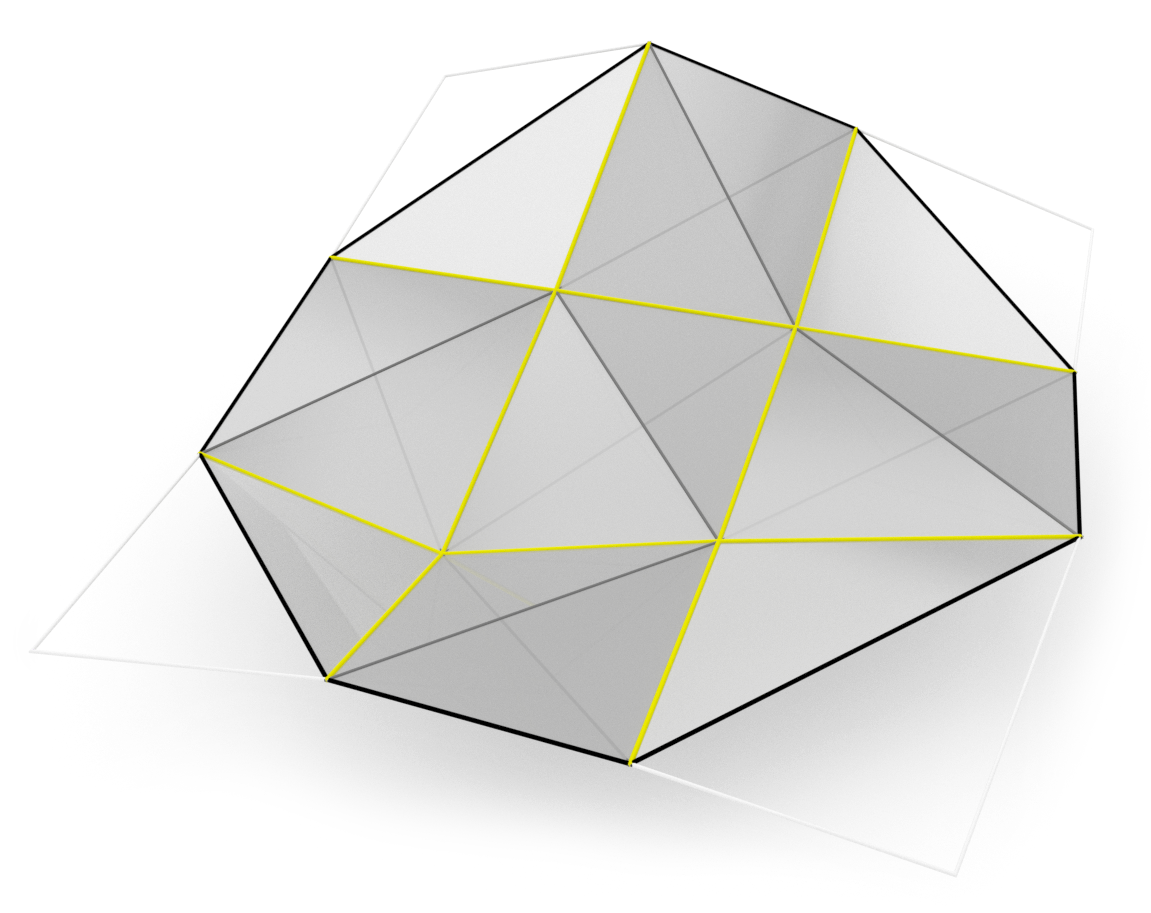}
\caption{\textcolor{black}{Left: For flexibility analysis of a $3 \times 3$
quad mesh (skew Kokotsakis mesh), we can replace the four corner quads by triangles. Right: For geometric analysis based on angles, diagonals are added to the remaining five quads so that they become tetrahedrons.  Triangles and tetrahedrons are considered rigid. The yellow edges where they are joined act as revolute joints.}} \label{fig:quadmesh}
\end{figure}

We propose a construction method for a \textcolor{black}{remarkable class of flexible skew Kokotsakis meshes.} We start with the notations for the \textcolor{black}{mesh} and the corresponding linkages of spherical quadrilaterals. \textcolor{black}{These spherical quadrilaterals represent the motion of $2 \times 2 $  submeshes.} Since the flexibility of a mesh is equivalent to the flexibility of a spherical linkage, the main focus of the study is the configuration space of the linkage. Firstly, we describe the configuration space of one orthodiagonal spherical quadrilateral in \cref{lma2.7} by means of a polynomial equation of two variables based on classical results.

\textcolor{black}{
\cref{sec3} starts with a brief discussion of Izmestiev's approach to the classification of the planar case based on a commutative diagram of branched covers associated with a flexible Kokotsakis mesh. \textcolor{black}{Proceeding towards the motion of $2 \times 3$ submeshes}, we introduce the notion of an involutive coupling, and in \cref{lma4.2}, we determine the conditions for the coupling of orthodiagonal quadrilaterals to be involutive. In \cref{lma3.4} we describe its configuration space as a rational equation.}

\textcolor{black}{
Finally, in \cref{sec5}, \textcolor{black}{we arrive at the description of flexible
$3 \times 3 $ meshes}. We show how to perform a matching of involutive couplings and define the class of orthodiagonal involutive skew Kokotsakis meshes. Moreover, we present algebraic reasoning to obtain the conditions of the class in terms of planar and dihedral angles and prove that they have a one-parameter set of real solutions, i.e. \textcolor{black}{that they are}  geometrically flexible.}

\textcolor{black}{In Section~\ref{sec:practice} we discuss the practical computation
of mechanisms. We illustrate it with an example for which we also built a physical
model from stainless steel and briefly address the fabrication process. }

\section{Configuration space of a spherical four-bar linkage}\label{sec2}
\textcolor{black}{
    Recall that the main object under consideration in this paper is a \emph{skew Kokotsakis mesh}, defined as a simply connected quad mesh with $3 \times 3$
    faces. We do not assume planarity of faces, but planar faces are allowed. 
     For our study, we confine to the essential part which is relevant for
     flexibility and therefore replace the four corner quads 
    by triangles.
 Moreover, we insert the two diagonals in each of the remaining five quads so
 that they become tetrahedrons (see Fig.~\ref{fig:quadmesh}). Also this resulting
 figure is called a skew Kokotsakis mesh. 
 The five tetrahedrons and four triangles are assumed to be rigid. The edges 
 in which they are joined act as revolute joints. }


\textcolor{black}{We call a skew Kokotsakis mesh \emph{flexible} if, after fixing one
face, it possesses a one-parametric family of isometric meshes in which
each face is related to its original position by a rigid body motion. } The problem of flexibility naturally admits a reformulation in terms of spherical geometry. By doing this, we remove excessive parameters such as lengths of edges and show that only angles play an important role for the property of flexibility.

\subsection{Notations}

\begin{figure}[ht]
        \centering
\scalebox{0.7}{\begin{overpic}[width=0.6\textwidth]{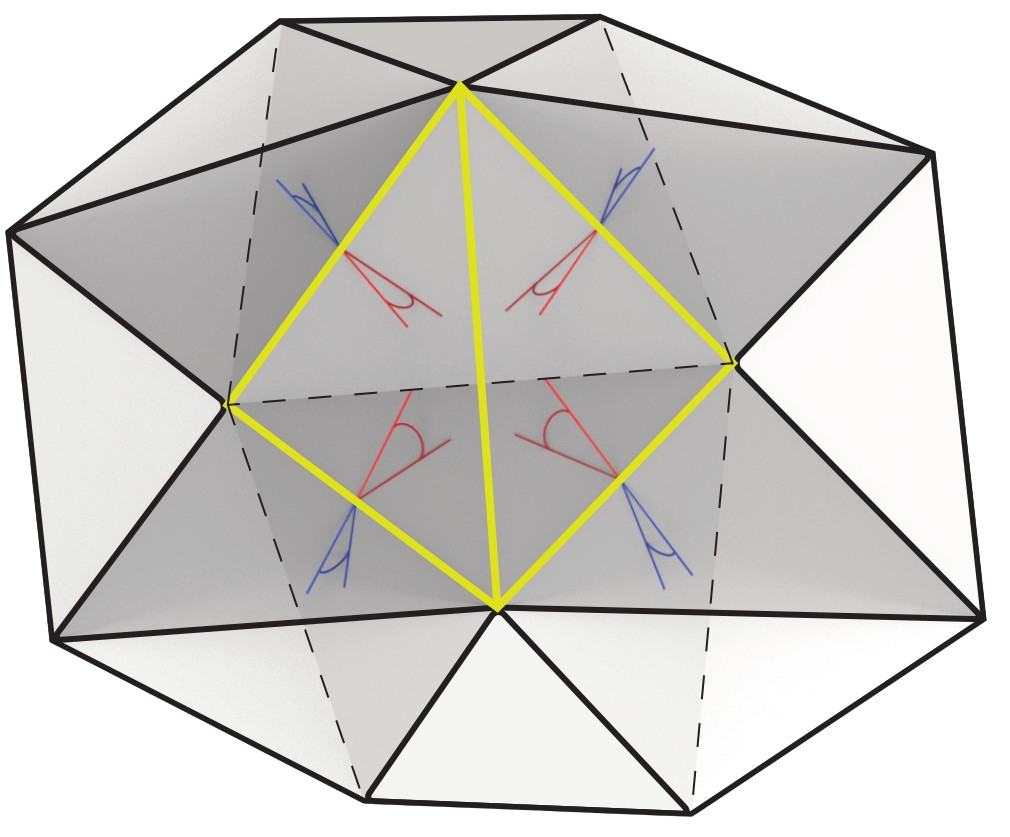}
    \put(43,77){\contour{white}{$A_3$}}
    \put(75,47){\contour{white}{$A_4$}}  
    \put(17,43){\contour{white}{$A_2$}}   
    \put(42,42){\contour{white}{$\tau_2$}}
    \put(50,41){\contour{white}{$\tau_1$}} 
    \put(41,49){\contour{white}{$\tau_3$}}
    \put(51,50){\contour{white}{$\tau_4$}}
    \put(30,30){\contour{white}{$\zeta_2$}}
    \put(28,58){\contour{white}{$\zeta_3$}}
    \put(62,61){\contour{white}{$\zeta_4$}}
    \put(66,31){\contour{white}{$\zeta_1$}}    
    \put(49,17){\contour{white}{$A_1$}}
    \put(70,0){\contour{white}{$B_1$}}
    \put(34,0){\contour{white}{$C_1$}}
    \put(1,20){\contour{white}{$B_2$}}
    \put(0,63){\contour{white}{$C_2$}}    
    \put(23,81){\contour{white}{$B_3$}}      
    \put(63,81){\contour{white}{$C_3$}}
    \put(93,67){\contour{white}{$B_4$}}    
    \put(98,20){\contour{white}{$C_4$}}  
\end{overpic}}
 \caption{Vertices and angles in a skew Kokotsakis mesh; $\tau_i$ denote \textcolor{black}{the constant dihedral} angles in the central tetrahedron (yellow) and $\zeta_j$ are the \textcolor{black}{constant dihedral} angles of the side tetrahedra.}\label{fig1}
\end{figure}
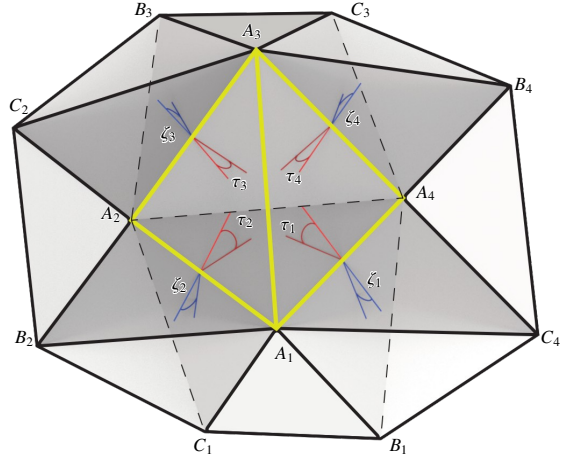

In this subsection we describe the notations of angles in order to transform the geometrical problem into algebra. As already mentioned, we interpret every skew quad as a rigid tetrahedron (see \cref{fig1}). The enumeration of vertices and angles is cyclic and takes values from $\{1\ldots 4\}$, for example, if $i = 4$ then $[A_{i}, A_{i + 1}, A_{i + 2}]$ denotes $[A_{4}, A_{1}, A_{2}]$ and if $i = 1$ $[A_{i - 2}, A_{i - 1}, A_{i}]$ denotes $[A_{3}, A_{4}, A_{1}]$.  

We denote by $\alpha_i, \beta_i, \gamma_i, \delta_i \in (0, \pi)$ the angles between edges as shown in \cref{fig2}. Each $\delta_i$ is the angle between edges $A_{i - 1}A_{i}, A_{i}A_{i + 1}$, $\beta_i$ is the angle between $B_{i}A_i, A_{i}C_{i}$, $\alpha_i$ is the angle between $C_{i}A_{i}, A_i A_{i+1}$ if $i = 1, 3$ and $B_i A_i, A_i A_{i - 1}$ if $i = 2, 4$, $\gamma_i$ is the angle between $B_i A_i, A_i A_{i - 1}$ if $i = 1, 3$ and $C_i A_i, A_i A_{i + 1}$ if $i = 2, 4$. They play the same role as the planar angles in~\cite{izmestiev-2017}.

In the analysis of the flexibility of the non-planar case, we have to consider the dihedral angles of all tetrahedrons as well. Denote by $a_i = A_{i - 1} A_i$, $b_i = B_i A_i$, $c_i = A_i C_i$, where for two points $A$ and $B$, $AB$ is the vector from point $A$ to point $B$. We fix the orientation such that $a_1, a_2, [a_1, a_2]$ are positively oriented, where $[a, b]$ is the cross product of vectors $a, b$. By $\tau_i, \zeta_i \in (-\pi, \pi)$ we denote the fixed oriented dihedral angles of the central tetrahedron $A_1 A_2 A_3 A_4$ and side tetrahedrons, respectively, as in \cref{fig1}. Each $\tau_i$ is the angle between faces $[A_{i - 2}, A_{i-1}, A_{i}], [A_{i - 1}, A_i, A_{i+1}]$, $\zeta_i$ is the angle between faces $[C_{i - 1}, A_{i - 1}, A_{i}], [A_{i - 1}, A_{i}, B_i]$. 

By $\phi_i,\psi_i \in (-\pi, \pi]$ we denote those oriented dihedral angles that change under flexion. We set
$$
\phi_i = \text{sign}(\phi_{i}^{*})\cdot(\pi - |\phi_{i}^*|), \quad \psi_i = \text{sign}(\psi_{i}^{*})\cdot(\pi - |\psi_{i}^*|),
$$
where $\phi_{i}^*$, $\psi_{i}^* \in (-\pi, \pi]$ are the oriented flexible dihedral angles (see \cref{fig2}) between faces $[A_{i - 1}, A_i, A_{i+1}], [A_{i - 1}, A_i, B_i]$ and $[A_{i - 2}, A_{i - 1}, A_{i}], [C_{i - 1}, A_{i - 1}, A_i]$, respectively, if $i = 1, 3$ and between faces $[A_{i - 2}, A_{i - 1}, A_{i}], [C_{i - 1}, A_{i - 1}, A_i]$ and $[A_{i - 1}, A_i, A_{i+1}]$, $[A_{i - 1}, A_i, B_{i}]$, respectively, if $i = 2, 4$. Changing the dihedral angles $\phi_i^{*}$ and $\psi_i^{*}$ corresponds to the flexibility of the Kokotsakis mesh.

We can calculate all dihedral angles with signs as follows:
\begin{equation*}
    \tau_i = \text{sign}(o_{\tau_i})\arccos((\overline{[a_{i - 1}, a_i]}, \overline{[a_i, a_{i + 1}]})), \quad \text{where} \quad o_{\tau_i} = \begin{cases}
    ([a_i, a_{i + 1}], a_{i + 2}), \text{ if } i = 1, 3\\
    ([a_{i - 1}, a_{i}],a_{i + 1}), \text{ if } i = 2, 4
    \end{cases},
\end{equation*}
\begin{equation*}
    \zeta_i = \text{sign}(o_{\zeta_i})\arccos((\overline{[a_{i}, b_i]}, \overline{[c_{i - 1}, a_{i}]})),\quad \text{where} \quad o_{\zeta_i} = 
    \begin{cases}
        ([a_i, b_i], c_{i-1}), \text{ if } i = 1, 3\\
        ([c_{i - 1}, a_i], b_i), \text{ if } i = 2, 4\\
    \end{cases},
\end{equation*}
and
\begin{equation*}
    \phi_i = 
    \begin{cases}
        \text{sign}(([a_i, a_{i + 1}], b_i))\arccos((\overline{[a_i, a_{i + 1}]}, \overline{[a_{i}, b_{i}]})), \text{ if } i = 1, 3\\
        \text{sign}(([a_{i - 1}, a_{i}], c_{i - 1}))\arccos((\overline{[a_{i - 1}, a_{i}]}, \overline{[c_{i- 1}, a_{i}]})), \text{ if } i = 2, 4
    \end{cases},
\end{equation*}
\begin{equation*}
    \psi_i =
    \begin{cases}
        \text{sign}(([a_{i - 1}, a_i], c_{i - 1}))\arccos((\overline{[a_{i - 1}, a_i]}, \overline{[c_{i - 1}, a_i]})), \text{ if } i = 1, 3\\
        \text{sign}(([a_{i}, a_{i + 1}], b_i)) \arccos((\overline{[a_i, a_{i + 1}]}, \overline{[a_i, b_i]})), \text{ if } i = 2, 4
    \end{cases},
\end{equation*}
where $(\cdot, \cdot)$ denotes the scalar product of vectors and $\overline{[a, b]} = \frac{[a, b]}{|[a, b]|}$ the normalization of vector $[a, b]$. These dihedral angles satisfy the equality $\psi_i = \phi_i + \tau_i + \zeta_i$.

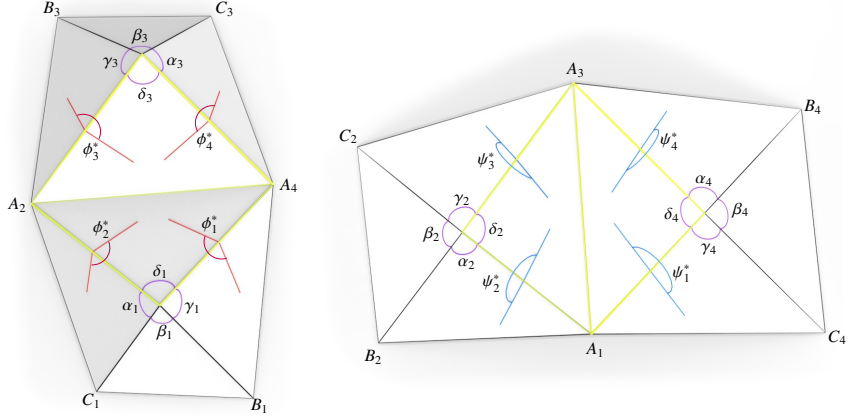
\begin{figure}[ht]
    \centering

\scalebox{0.65}{
    \begin{tikzpicture}
            \node[anchor=south west,inner sep=0] at (0,0) {\includegraphics[width=\textwidth]{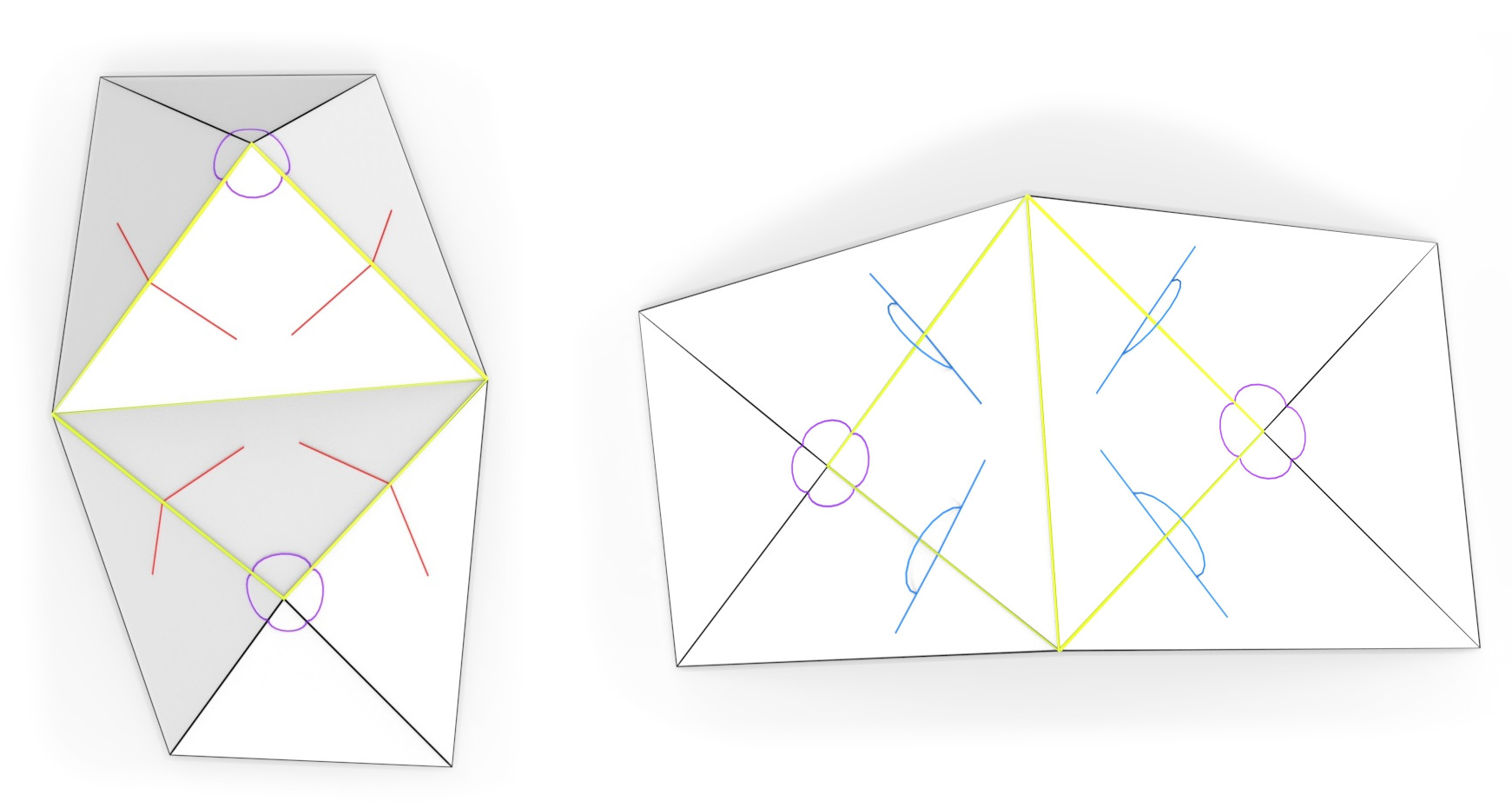}};

    \draw (3.2, 3.05) node {$\delta_1$};
    \draw (3.85, 2.3) node {$\gamma_1$};
    \draw (3.3, 1.8) node {$\beta_1$};
    \draw (2.6, 2.3) node {$\alpha_1$};
    
    \draw (2.8, 7.85) node {$\beta_3$};
    \draw (3.5, 7.3) node {$\alpha_3$};
    \draw (2.9, 6.65) node {$\delta_3$};
    \draw (2.2, 7.3) node {$\gamma_3$};

    \draw (4.2, 4) node {$\phi_1^*$};
    \draw[purple] (4.1, 3.77) arc(160:288:10pt);
    \draw (2, 3.9) node {$\phi_2^*$};
    \draw[purple] (2.15, 3.65) arc(20:-93:10pt);
    \draw (1.8, 5.5) node {$\phi_3^*$};
    \draw[purple] (1.95, 5.75) arc(-25:110:10pt);
    \draw (4.1, 5.7) node {$\phi_4^*$};
    \draw[purple] (3.95, 5.9) arc(210:83:10pt);

    \draw (1.8, 0.4) node {$C_1$};
    \draw (5.2, 0.3) node {$B_1$};
    \draw (0.3, 4.4) node {$A_2$};
    \draw (1, 8.4) node {$B_3$};
    \draw (4.5, 8.4) node {$C_3$};
    \draw (5.8, 4.8) node {$A_4$};

    \draw (10, 3.9) node {$\delta_2$};
    \draw (9.3, 4.5) node {$\gamma_2$};
    \draw (8.7, 3.8) node {$\beta_2$};
    \draw (9.4, 3.2) node {$\alpha_2$};
    
    \draw (14.95, 4.3) node {$\beta_4$};
    \draw (14.3, 3.5) node {$\gamma_4$};
    \draw (13.55, 4.2) node {$\delta_4$};
    \draw (14.2, 4.9) node {$\alpha_4$};

    \draw (12, 1.45) node{$A_1$};
    \draw (7.5, 1.3) node {$B_2$};
    \draw (7, 5.8) node {$C_2$};
    \draw (11.6, 7.15) node{$A_3$};
    \draw (16.9, 1.7) node {$C_4$};
    \draw (16.4, 6.4) node {$B_4$};
    
    \draw (13.75, 3) node {$\psi_1^*$};
    \draw (13.5, 5.7) node {$\psi_4^*$};
    \draw (9.9, 2.8) node {$\psi_2^*$};
    \draw (9.8, 5.3) node {$\psi_3^*$};
    
    \end{tikzpicture}
}

    \caption{Dihedral angles $\phi_i^*,\psi_i^*$ change during flexion of a skew Kokotsakis mesh; $(\alpha_i, \beta_i, \gamma_i, \delta_i)$ denote constant planar angles.}\label{fig2}
\end{figure}

\subsection{Spherical image}\label{sec3.2}

Firstly, we follow the classical approach of associating it with a movable spherical linkage (cf. \cite{izmestiev-2017}, \cite{nawratil-2010}, \cite{nawratil-2011}, \cite{nawratil-2012}, \cite{Nawratil2022} and \cite{stachel2010}). 
For each of the four interior vertices $A_1, A_2, A_3, A_4$ consider its spherical image, that is, the intersection of the cone of adjacent faces with a unit sphere centered at the vertex. This yields four spherical quadrilaterals $Q_i$ with side lengths $\alpha_i, \beta_i, \gamma_i, \delta_i$ in this cyclic order.  In the works mentioned above, the faces of a Kokotsakis mesh are planar. Therefore the spherical images of two adjacent vertices are {\it coupled} by means of a common dihedral angle. However, since we have
skew quad faces, we have to consider a generalization of that construction. Spherical images of two adjacent vertices are {\it coupled} if the difference of their dihedral angles in the common vertices is constant during all deformations. Thus, for every edge $A_i A_{i+1}$ we have a scissors-like coupling of two spherical quadrilaterals, see \cref{fig3}.

\textcolor{black}{A skew Kokotsakis mesh is flexible if and only if the spherical linkage that is formed by the closed chain of four coupled spherical images corresponding to the four edges of the base quad is flexible.} Thus, in order to study flexible Kokotsakis meshes, it suffices to consider all such flexible spherical linkages.

We start by recalling classical results about the configuration space of a spherical quadrilateral for the first quadrangle $Q_1$ in \cref{fig3}.

A spherical quadrilateral $Q_1$ with given side lengths $(\alpha_1, \beta_1, \gamma_1, \delta_1)$ is uniquely determined by the values of two adjacent angles; on the other hand, these angles satisfy a certain relation.
By performing the substitution
\begin{equation}\label{eq1}
    x_1 = \tan\frac{\phi_1}{2}, \quad x_2 = \tan\frac{\phi_2}{2}, 
\end{equation}
where $\phi_1$ and $\phi_2$ are as in \cref{fig3}, this leads to a polynomial equation for $x_1$ and $x_2$. 

\begin{lemma}[\cite{bricard1897}]
The configuration space of a spherical quadrilateral $Q_1$ with side lengths $\alpha_1$, $\beta_1$, $\gamma_1$, $\delta_1$ in this cyclic order is the solution of the equation
\begin{equation}\label{eq2}
P_1(x_{2}, x_1) = c_{22} x_{2}^2 x_1^2 + c_{20} x_{2}^2 + c_{02} x_1^2 + 2c_{11} x_{2} x_1 + c_{00} = 0,
\end{equation}
where
$$
\begin{matrix}
c_{22} = \sin \frac{\alpha_1 + \beta_1 + \gamma_1 - \delta_1}{2} \sin \frac{\alpha_1 - \beta_1 + \gamma_1 - \delta_1}{2}, \\
c_{20} = \sin \frac{\alpha_1 - \beta_1 - \gamma_1 - \delta_1}{2} \sin \frac{\alpha_1 + \beta_1 - \gamma_1 - \delta_1}{2}, \\
c_{02} = \sin \frac{\alpha_1 + \beta_1 - \gamma_1 + \delta_1}{2} \sin \frac{\alpha_1 - \beta_1 - \gamma_1 + \delta_1}{2}, \\
c_{11} = -\sin\alpha_1 \sin \gamma_1, \\
c_{00} = \sin \frac{\alpha_1 - \beta_1 + \gamma_1 + \delta_1}{2} \sin \frac{\alpha_1 + \beta_1 + \gamma_1 + \delta_1}{2}.
\end{matrix}
$$
\end{lemma}
The proof can be found in \cite{stachel2010}. We view equation \cref{eq2} as an equation in two projective variables $x_1, x_2 \in \mathbb{C}\text{P}^1$, to incorporate the value $\infty$ for $\phi_i=\pi$. By $Z_1$ we denote the solution set of \cref{eq2} in $(\mathbb{C}\text{P}^1)^2$
$$
Z_1 = \{(x_2, x_1) \in \mathbb{C}\text{P}^1 \times \mathbb{C}\text{P}^1| P_1(x_2, x_1) = 0\}.
$$

The same result holds for the other spherical quads $Q_i$, which means that the configuration space of a Kokotsakis mesh equals the set of solutions for the system
\begin{equation}\label{eq3}
P_1(x_2, x_1) = 0, \quad P_2(y_2, y_3) = 0, \quad P_3(x_4, x_3) = 0, \quad P_4(y_4, y_1) = 0,
\end{equation}
where $x_i = \tan \frac{\phi_i}{2}$,  $y_i = \tan \frac{\psi_i}{2}$. \textcolor{black}{Thus we have the following lemma
\begin{lemma}
    A skew Kokotsakis mesh with a quad base is flexible if and only if the system of polynomial equations~\cref{eq3} has a family of one-parameter solutions over the reals.
\end{lemma}}
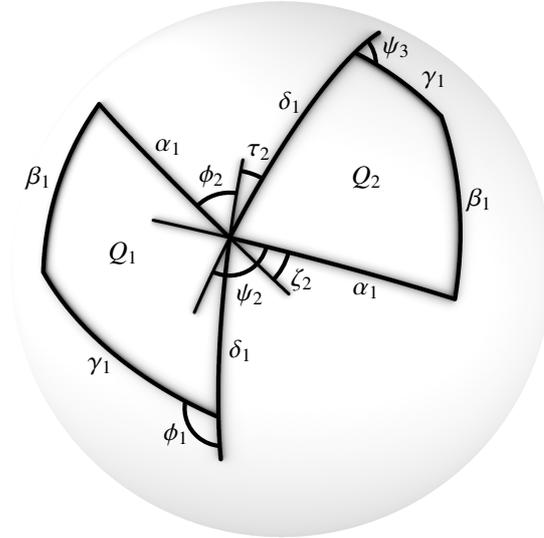
\begin{figure}[ht!]
    \centering
        \begin{tikzpicture}
        \node[anchor=south west,inner sep=0] at (-3.73, -4.63) {\includegraphics[width=0.5\textwidth]{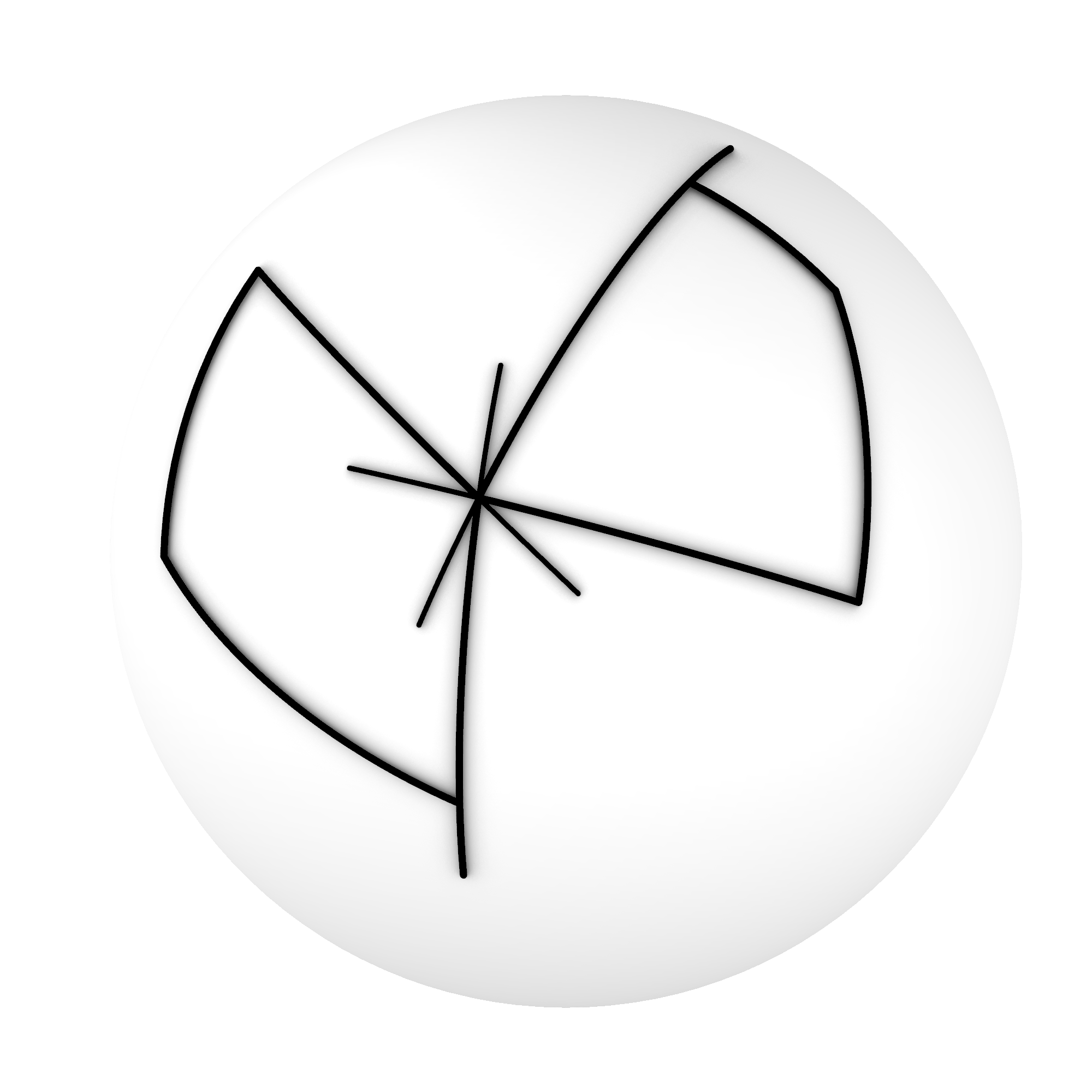}};

        \draw[ultra thick] (245:0.5) arc (245:345:0.5);
        \draw (290:0.8) node {$\psi_2$};
        
        \draw[ultra thick] (315:0.8) arc (315:345:0.8);
        \draw (330:1.1) node {$\zeta_2$};
        
        \draw[ultra thick] (60:0.9) arc (60:80:0.9);
        \draw (71:1.2) node {$\tau_2$};
        
        \draw[ultra thick] (82:0.6) arc (82:135:0.6);
        \draw (106:0.85) node {$\phi_2$};
        
        \draw[ultra thick] (255:2.25) arc (170:265:0.5);
        \draw (255:2.7) node {$\phi_1$};
        
        \draw[ultra thick] (50:3) arc (0:40:0.5);
        \draw (49:3.35) node {$\psi_3$};

        \draw (-0.8, 1.2) node {$\alpha_1$};
        \draw (-2.5, 0.8) node {$\beta_1$};
        \draw (-1.7, -1.7) node {$\gamma_1$};
        \draw (0.15, -1.5) node {$\delta_1$};

        \draw (1.8, -0.7) node {$\alpha_1$};
        \draw (3.3, 0.5) node {$\beta_1$};
        \draw (2.7, 2.1) node {$\gamma_1$};
        \draw (0.8, 1.8) node {$\delta_1$};

        \draw (-1.4, -0.2) node {$Q_1$};
        \draw (1.8, 0.8) node {$Q_2$};
    \end{tikzpicture}
    \caption{A coupling $(Q_1, Q_2)$ of two spherical quadrilaterals.}
    \label{fig3}
\end{figure}    
\textcolor{black}{Notice that due to $\psi_i = \phi_i + \tau_i + \zeta_i$, the variables $y_i$ can be expressed by $x_i$ in the following form
\begin{equation}\label{eq8}
y_i(x_i) = \frac{x_i + F_i}{1 - F_i x_i},
\end{equation}
where $F_i = \tan\frac{\tau_i + \zeta_i}{2} \in \mathbb{R}\text{P}^1$.} 
\begin{definition}\label{def-pseudoplanar}
A skew Kokotsakis mesh is called pseudo-planar if $F_i = 0$ for all $i$.
\end{definition}
The definition is justified by the fact that equations which define the configuration space of the Kokotsakis mesh (system \cref{eq3}) coincide with the planar case.

The approach used by Izmestiev in \cite{izmestiev-2017} for the case of planar faces (where $y_i = x_i$) was to find conditions when algebraic sets $R_{12}(x_1, x_3) = \text{res}_{x_2}(P_1, P_2) = 0$ and $R_{34}(x_1, x_3) = \text{res}_{x_4}(P_3, P_4) = 0$ have a common irreducible component, where $\text{res}_{x_k}(P_i, P_j)$ stands for a resultant with respect to $x_k$. We will follow the same approach in \cref{sec3}.

\textcolor{black}{Note that functions
$$
P_{2}(y_2, y_3) = \frac{\tilde{P}_2(x_2, x_3)}{(1 - F_2 x_2)(1 - F_3 x_3)}, \quad P_{4}(y_4, y_1) = \frac{\tilde{P}_4(x_4, x_1)}{(1 - F_4 x_4)(1 - F_1 x_1)}
$$
are rational in $x_i$, and by $R_{12}(x_1, x_3), R_{34}(x_1, x_3)$ we mean $\text{res}_{x_2}(P_1, \tilde{P}_2), \text{res}_{x_4}(P_3, \tilde{P}_4)$, respectively.}

\subsection{Orthodiagonal quadrilateral}
\textcolor{black}{In this subsection, we remind the reader of the main results about orthodiagonal quadrilaterals on the sphere. Orthodiagonal quadrilaterals admit a simple representation of their configuration space which is described in \cref{lma2.7}.}

\textcolor{black}{Let $Q$ be a spherical quadrilateral with side lengths $\alpha, \beta, \gamma, \delta$ in this cyclic order.
\begin{definition}
A spherical quadrilateral is said to be {\it orthodiagonal}, if its diagonals are orthogonal, see \cref{fig:orthodiagonal-quads}.
\end{definition}}

The orthodiagonality is equivalent to the following identity.
    \begin{lemma}[Lemma 6.3, \cite{EI20}]\label{lma5.3}
        The diagonals of a spherical quadrilateral with side lengths $\alpha, \beta, \gamma, \delta$ (in this cyclic order) are orthogonal if and only if its side lengths satisfy the relation
        \begin{equation}\label{eq33}
            \cos \alpha \cos \gamma = \cos \beta \cos \delta.
        \end{equation}
    \end{lemma}

The orthodiagonality property is quite remarkable as it also implies an existence property:
    \begin{lemma}[Lemma 6.4, \cite{EI20}]
        Let $\alpha, \beta, \gamma, \delta \in (0, \pi)$ satisfy $\cos\alpha \cos \gamma = \cos \beta \cos \delta$. Then there exists a spherical orthodiagonal quadrilateral with side lengths $\alpha, \beta, \gamma, \delta$.
    \end{lemma}
\begin{lemma}[Lemma 4.12, \cite{izmestiev-2017}]\label{lma2.5}
\textcolor{black}{If $Q_1$ is an orthodiagonal quadrilateral then it is one of the two following types}
 \begin{itemize}
        \item \textcolor{black}{of {\it elliptic} type, i.e. equation
            \begin{equation*}\label{eq4}
                \alpha \pm \beta \pm \gamma \pm \delta = 0\quad(\text{mod }2\pi)
            \end{equation*}
        has no solutions;}
        \item \textcolor{black}{a {\it deltoid}, i.e it has two pairs of equal adjacent sides, and an {\it antideltoid}, if it has two pairs of adjacent sides complementing each other to $\pi$.}
    \end{itemize}
\end{lemma}

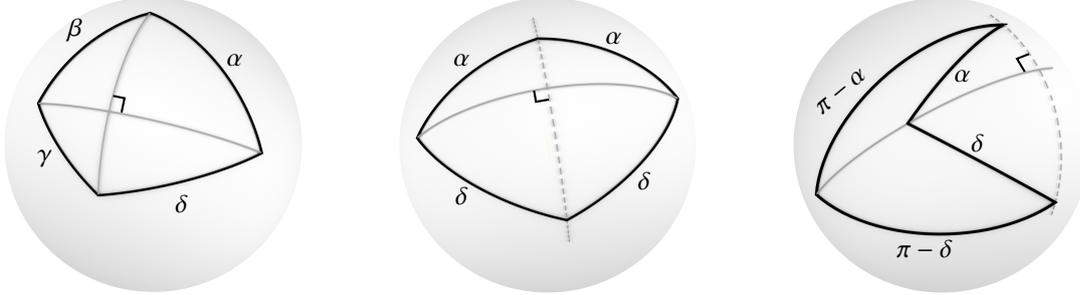
\begin{figure}[htb]
    \centering
    \begin{tikzpicture}
        \node[anchor=south west,inner sep=0] at (0,0) {\includegraphics[width=0.3\textwidth]{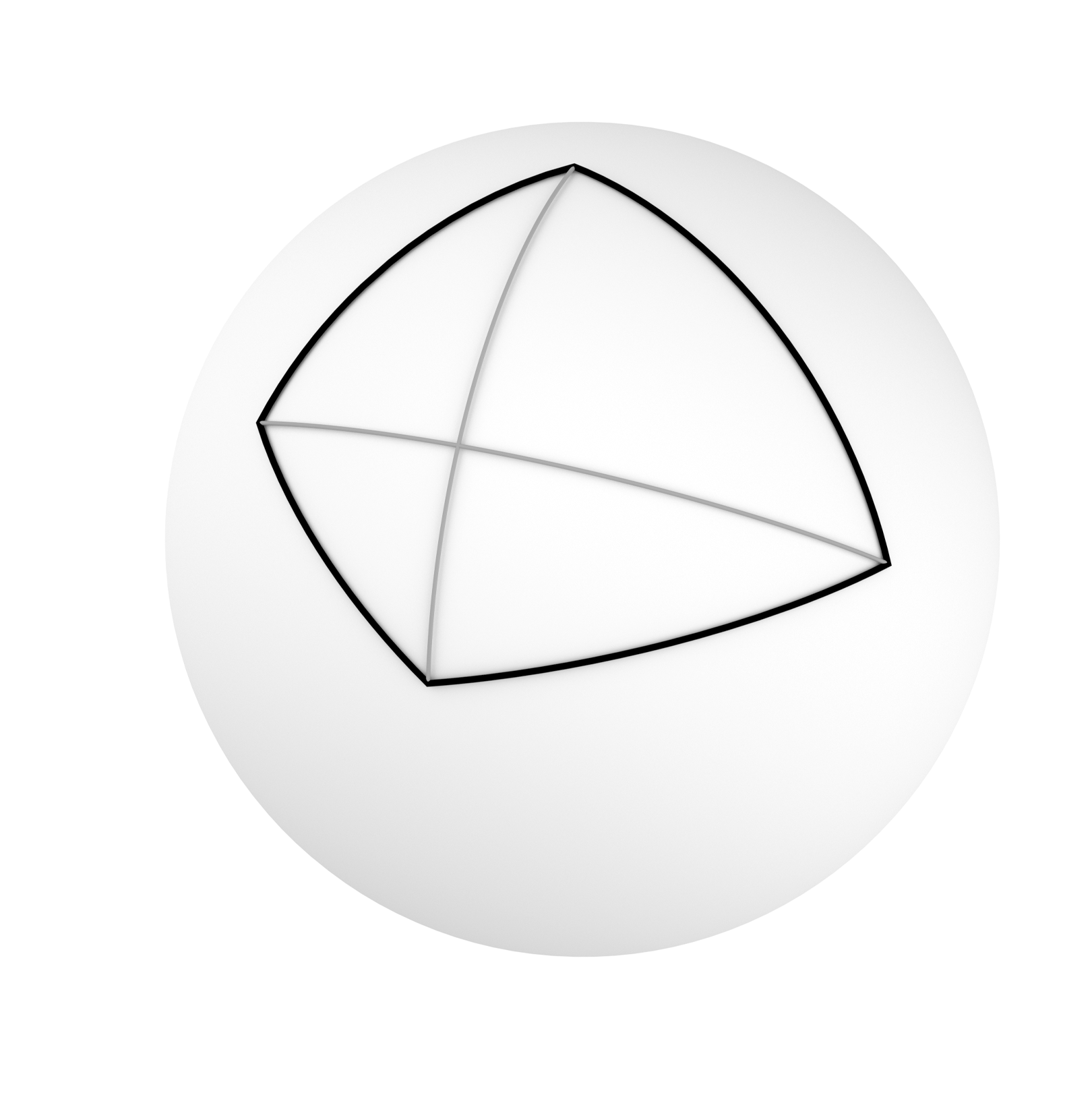}};
        \draw[thick] (2.31, 3.01) -- (2.35, 3.2) -- (2.2, 3.23);
        \draw (3.8, 3.7) node {$\alpha$};
        \draw (1.7, 4.1) node {$\beta$};
        \draw (1.3, 2.4) node {$\gamma$};
        \draw (3.1, 1.8) node {$\delta$};    
    \end{tikzpicture}
    \begin{tikzpicture}
        \node[anchor=south west,inner sep=0] at (0,0) {\includegraphics[width=0.3\textwidth]{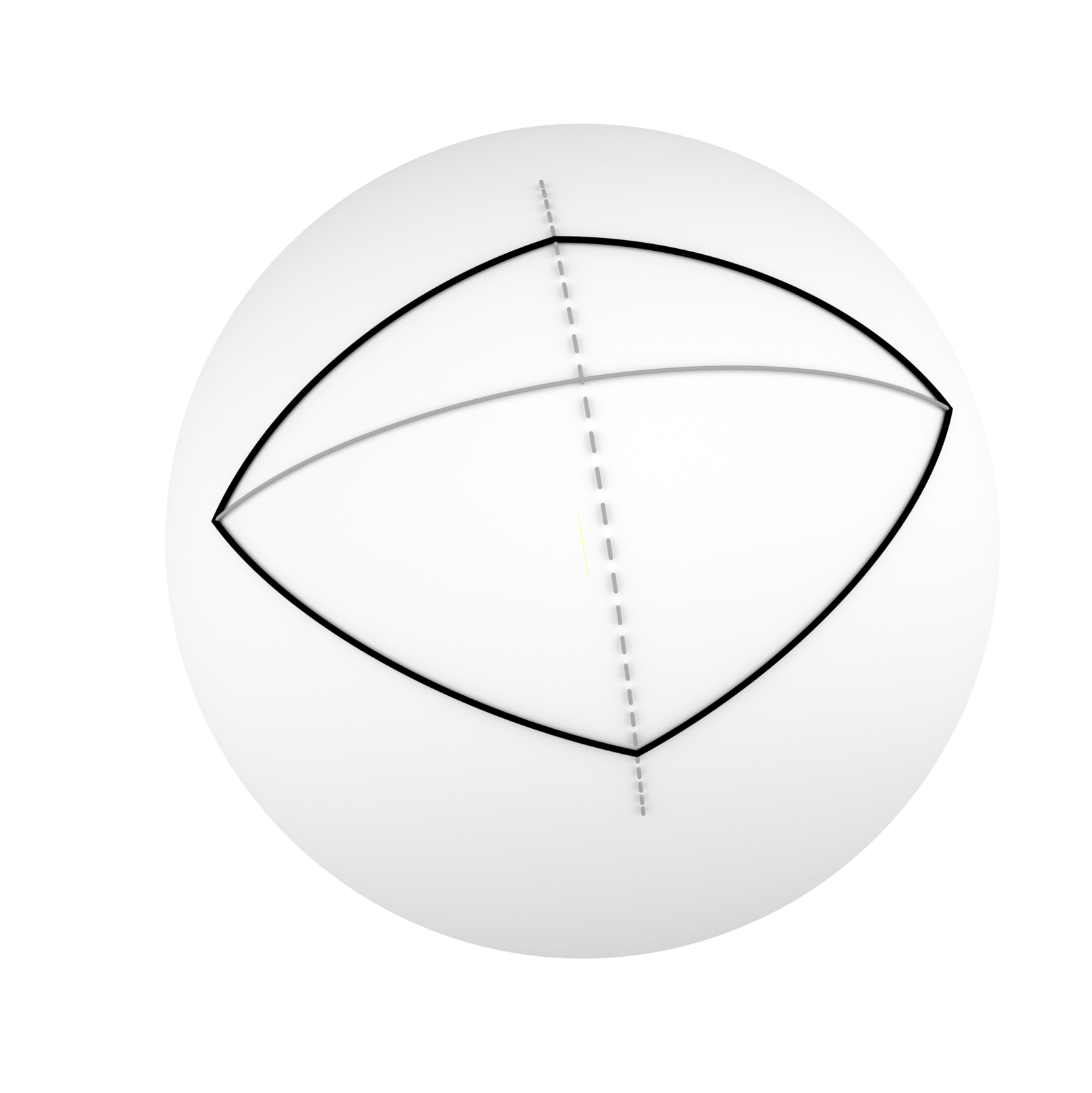}};
        \draw[thick, rotate around ={10:(2.55, 3.3)}] (2.55, 3.3) -- (2.55, 3.15) -- (2.72, 3.15);
        \draw (3.6, 4) node {$\alpha$};
        \draw (1.6, 3.7) node {$\alpha$};
        \draw (1.6, 1.9) node {$\delta$};
        \draw (4, 2.1) node {$\delta$};    
    \end{tikzpicture}
    \begin{tikzpicture}
        \node[anchor=south west,inner sep=0] at (0,0) {\includegraphics[width=0.3\textwidth]{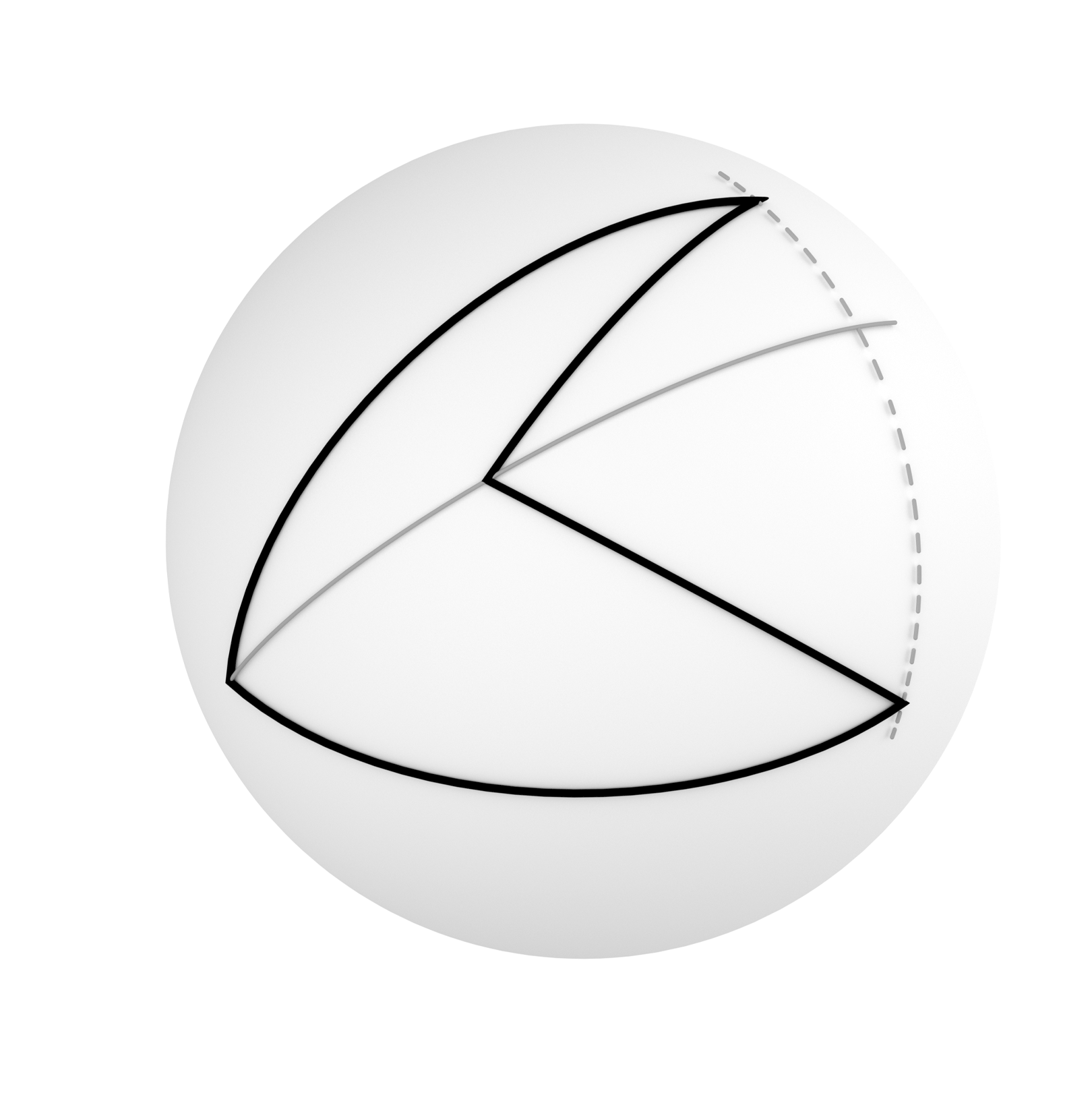}};
        \draw[thick, rotate around = {25:(3.82, 3.55)}] (3.82, 3.55) -- (3.8, 3.75) -- (3.97, 3.74);
        \draw (3, 3.5) node {$\alpha$};
        \draw (1.4, 3.3) node[rotate=40] {$\pi - \alpha$};
        \draw (2.5, 1.2) node {$\pi - \delta$};
        \draw (3.2, 2.6) node {$\delta$};    
    \end{tikzpicture}  
\caption{\textcolor{black}{Left: general orthodiagonal quad.
Middle: For a deltoid, one diagonal is a symmetry axis (dashed). The vertices on it are the apices. Right: Reflecting a vertex of a deltoid that is not an apex at the center of the sphere, one obtains an antideltoid. Its
apices agree with the ones of the deltoid from which it
has been derived.}
} \label{fig:orthodiagonal-quads}
\end{figure}
The case when $\alpha_i = \beta_i = \gamma_i = \delta_i = \frac{\pi}{2}$ is excluded, as it leads only to trivial deformations. We refer to a vertex of a quadrilateral by naming the two sides incident to it. We say that an (anti)deltoid has \emph{apices} $\alpha \delta$ and $\beta \gamma$, if $\alpha = \delta$, $\beta = \gamma$ or $\alpha + \delta = \pi = \beta + \gamma$. \textcolor{black}{For an
illustration, see Fig.~\ref{fig:orthodiagonal-quads}.}

\begin{definition}
    An (anti)deltoid $Q_1$ is said to be {\it frontally} coupled with $Q_2$ if the common vertex of $Q_1$ and~$Q_2$ is an apex of $Q_1$. Otherwise, $Q_1$ is said to be {\it laterally} coupled with $Q_2$.    
\end{definition}
\begin{definition}\label{def2.6}
Let $Q$ be an orthodiagonal quadrilateral. Following \cite{izmestiev-2017}, we define the involution factors at each of its vertices, excluding the apices if $Q$ is an (anti)deltoid, as follows.

The involution factor at the vertex $\alpha \delta$ is
$$
\lambda := 
    \begin{cases}
        \frac{\tan(\delta) + \tan(\alpha)}{\tan(\delta) - \tan(\alpha)} \text{, if } \alpha \neq \frac{\pi}{2} \text{ or } \delta \neq \frac{\pi}{2},\\
        \frac{\cos(\beta) + \cos(\gamma)}{\cos(\beta) - \cos(\gamma)} \text{, if } \alpha = \delta = \frac{\pi}{2}.
    \end{cases}
$$

Similarly, the involution factor at the vertex $\gamma \delta$ is
$$
\mu :=
    \begin{cases}
        \frac{\tan(\delta) + \tan(\gamma)}{\tan(\delta) - \tan(\gamma)} \text{, if } \gamma \neq \frac{\pi}{2} \text{ or } \delta \neq \frac{\pi}{2},\\
        \frac{\cos(\beta) + \cos(\alpha)}{\cos(\beta) - \cos(\alpha)} \text{, if } \gamma = \delta = \frac{\pi}{2}.
    \end{cases}
$$

Besides, for an orthodiagonal quadrilateral of elliptic type we put 
$$
\nu :=
\begin{cases}
\frac{(\lambda - 1)(\mu - 1)}{\cos(\delta)} \text{, if } \delta \neq \frac{\pi}{2},\\
2(\mu - 1) \tan(\alpha) \text{, if } \delta = \gamma = \frac{\pi}{2},\\
2(\lambda - 1) \tan(\gamma) \text{, if } \delta = \alpha = \frac{\pi}{2}.
\end{cases}
$$

For an (anti)deltoid with apex $\gamma \delta$ we put
$$
\xi :=
    \begin{cases}
        \frac{\lambda - 1}{\cos(\delta)} \text{, if } \delta \neq \frac{\pi}{2},\\
        2 \tan(\alpha) \text{, if } \delta = \gamma = \frac{\pi}{2}.
    \end{cases}
$$
\end{definition}
The involution factors are well-defined real numbers different from 0. For example, if we consider
$\delta = \alpha = \gamma = \frac{\pi}{2}$, then 
$$\nu = 2(\mu - 1) \tan(\alpha) = 2\left(\frac{\cos\beta + \cos\alpha}{\cos\beta - \cos\alpha} - 1\right)\frac{\sin\alpha}{\cos\alpha} = \frac{4\cos\alpha}{\cos\beta - \cos\alpha} \frac{\sin\alpha}{\cos\alpha} = \frac{4}{\cos\beta}=2(\lambda - 1) \tan(\gamma).$$
If $\alpha = \frac{\pi}{2}$ and $\delta \neq \frac{\pi}{2}$, then $\lambda = \frac{\infty}{-\infty} = -1$.

\textcolor{black}{These parameters allow us
to abbreviate the equation of the configuration space of an orthodiagonal quadrilateral
in the following way.} 
\begin{lemma}[Corollary 4.15, \cite{izmestiev-2017}]\label{lma2.7}
The configuration space of an orthodiagonal quadrilateral~$Q_1$ has the equation 
\begin{equation}\label{eq5}
(x_2 + \lambda_1 x_2^{-1})(x_1 + \mu_1 x_1^{-1}) = \nu_1 \text{, if $Q_1$ is elliptic} 
\end{equation}
\begin{equation}
x_2 + \lambda_1 x_2^{-1}= \xi_1 x_1^{n_1} \text{, if $Q_1$ is an (anti)deltoid with apex } \alpha \beta
\end{equation}
Here $n_1 = 1$, if $Q_1$ is a deltoid, and $n_1 = -1$, if $Q_1$ is an antideltoid; $\lambda_1$, $\mu_1$, $\nu_1$, $\xi_1$ are as in~\cref{def2.6}.
\end{lemma}
\textcolor{black}{The main property of the configuration space $Z_1$ of an elliptic orthodiagonal quadrilateral is that it admits two involutions
    \begin{equation*}
        i_1 :Z_1 \to Z_1, \quad j_1:Z_1 \to Z_1
    \end{equation*}
    \begin{equation*}
        i_1:(x_2, x_1) \to (x_2, x'_1), \quad j_1:(x_2, x_1) \to (x'_2, x_1)   
    \end{equation*}
    \begin{equation*}
        i_1(x_2, x_1) = (x_2, \mu_1 x_1^{-1}), \quad j_1(x_2, x_1) = (\lambda_1 x_2^{-1}, x_1) 
    \end{equation*}
and an (anti)deltoid admits only one of the involutions depending on what involution factor is defined.
Geometrically, involutions $i_1,j_1$ act by folding the quadrilateral along one of its diagonals, see \cref{fig6}.} 

\textcolor{black}{For convenience, we write $j_1(x_2) = x'_2$ instead of $j_1(x_2, x_1) = (x_2', x_1)$, and analogously for~$i_1$.}
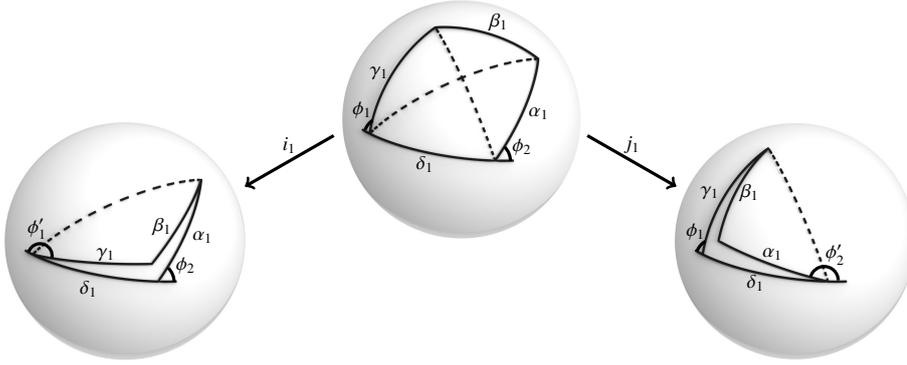
\begin{figure}[ht!]
    \centering
    \scalebox{0.75}{
        \begin{tikzpicture}
            \node[anchor=south west,inner sep=0] at (0,0) {\includegraphics[width=\textwidth]{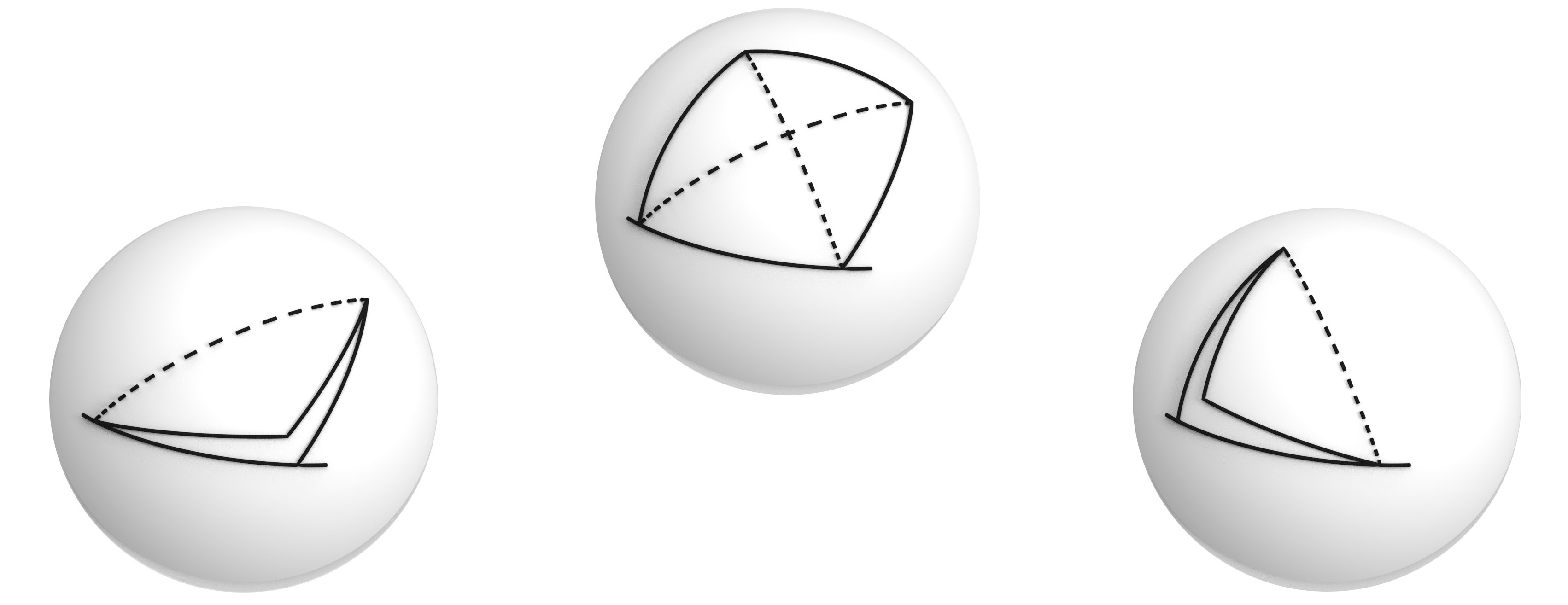}};
            \draw[<-, ultra thick, rotate around = {30:(6.3, 4.1)}] (4.5, 4.1) -- (6.3, 4.1)  node[above, midway] {$i_1$};
            \draw[->, ultra thick, rotate around = {-30:(10.75, 4.1)}] (10.75, 4.1) -- (12.55, 4.1) node[above, midway] {$j_1$};
    
            \draw (7.9, 3.6) node {$\delta_1$};
            \draw (7.1, 5.2) node {$\gamma_1$};
            \draw (9.2, 6.1) node {$\beta_1$};
            \draw (9.9, 4.5) node {$\alpha_1$};
            \draw[ultra thick] (6.85, 4.2) arc(160:124:10pt);
            \draw (6.8, 4.55) node {$\phi_1$};
            \draw[ultra thick] (9.4, 3.65) arc(0:45:10pt);
            \draw (9.6, 3.9) node {$\phi_2$};
            
            \draw (2, 1.4) node {$\delta_1$};
            \draw[ultra thick] (0.95, 2.05) arc(160:-10:6pt);
            \draw (1.1, 2.45) node {$\phi_1'$};
            \draw[ultra thick] (3.5, 1.51) arc(0:45:10pt);
            \draw (3.7, 1.8) node {$\phi_2$};
            \draw (2.3, 2) node {$\gamma_1$};
            \draw (3.3, 2.5) node {$\beta_1$};
            \draw (4, 2.3) node {$\alpha_1$};
    
            \draw (13.7, 1.5) node {$\delta_1$};
            \draw[ultra thick] (15.15, 1.54) arc(0:165:7pt);
            \draw (15.1, 2) node {$\phi_2'$};
            \draw[ultra thick] (12.67, 2.05) arc(160:121:10pt);
            \draw (12.65, 2.4) node {$\phi_1$};
            \draw (12.9, 3.1) node {$\gamma_1$};
            \draw (13.6, 3.1) node {$\beta_1$};
            \draw (14, 2) node {$\alpha_1$};
        \end{tikzpicture}    
    }
    \caption{\textcolor{black}{Action of involutions on a spherical quadrilateral.}}
    \label{fig6}
\end{figure}
\section{Composition of four-bar linkages to a mechanism}\label{sec3}
Our goal is to form a flexible mechanism by means of four-bar linkages. Firstly, we formulate an algebraic definition of the configuration space of coupled four-bar linkages (see Section~\ref{coupled-linkages}). Secondly, we show that orthodiagonal quadrilaterals are allowed to connect to each other by means of a so-called involutive coupling with a clear algebraic meaning (see Section~\ref{involutive-coupling}). Finally, in Section~\ref{matching-involutive}, we show how to match them into a mechanism using properties of orthodiagonal quadrilaterals with involutive coupling.

\subsection{Configuration space of two coupled four-bar linkages}\label{coupled-linkages}
Now we can discuss the coupling of two four-bar linkages as shown in \cref{fig3}. Its configuration space  is described as the following set of solutions
\begin{equation}
\widehat{Z}_{12} = \{(x_1, x_2, y_2, y_3) \in (\mathbb{C}\text{P}^1)^4|(x_2, x_1) \in Z_1, (y_2, y_3) \in \widehat{Z}_2\}
\end{equation}
where $Z_1$ and $\widehat{Z}_2$ are solutions sets of $P_1(x_2, x_1) = 0$ and $P_2(y_2, y_3) = 0$, respectively. \textcolor{black}{Since the relation \cref{eq8} is just a M\"obius transformation of $\mathbb{C}\text{P}^1$ to $\mathbb{C}\text{P}^1$, the complex algebraic curve}
$$
\widehat{Z}_2 = \{(y_2, y_3) \in (\mathbb{C}\text{P}^1)^2| P_2(y_2, y_3) = 0\}
$$
can be identified with
$$
Z_2 = \{(x_2, x_3) \in (\mathbb{C}\text{P}^1)^2| y_2 = \frac{x_2 + F_2}{1 - F_2 x_2}, y_3 = \frac{x_3 + F_3}{1 - F_3 x_3} \text{, where } (y_2, y_3) \in \widehat{Z}_2\}.
$$
Therefore, we redefine $Z_{12}$ as the following curve,
\begin{equation}\label{eq9}
Z_{12} = \{(x_1, x_2, x_3) \in (\mathbb{C}\text{P}^1)^3|(x_1, x_2) \in Z_1, (x_2, x_3) \in Z_2\}.
\end{equation}

The projection of $Z_{12}$ to the $(x_1, x_3)$-plane is the zero set of the resultant $R_{12} = \text{res}_{x_2}(P_1, \tilde{P}_2) = 0$. The spaces $Z_{3}, \widehat{Z}_4$ and $Z_{34}, \widehat{Z}_{34}$ are defined analogously for $P_3(x_4, x_3), P_4(y_4, y_1)$.

\textcolor{black}{An important point in the classification of $3 \times 3$ building blocks \cite{izmestiev-2017} is to study all possible couplings and characterize them according to the property of reducibility and involutivity. Then, in order to compose a matching, two couplings must have a common component. A component of $Z_i$ is called {\it trivial}, if it has the form $x_i = \text{const}$ or $x_{i + 1} = \text{const}$. For non-trivial components, the restrictions of the projection maps $Z_{12} \to Z_{i}$, where $i = 1, 2$ are branched covers between Riemann surfaces. If the solution set $Z_{all}$ of the system \cref{eq3} is one dimensional, then the map $Z_{all} \to Z_{12}$ is also a branched cover. Izmestiev based his classification \cite{izmestiev-2017} on a study of a commutative diagram of these branched covers. All maps in this diagram are at most two-fold. Since $\widehat{Z}_{12}$ and $\widehat{Z}_{34}$ are homeomorphic to $Z_{12}$ and $Z_{34}$, as it was shown before, the commutative diagram for the generalized case should coincide with the planar case. Therefore, we can determine the {\it orthodiagonal involutive} type similarly to Definition 5.7 in \cite{izmestiev-2017}.}
\subsection{Involutive coupling}\label{involutive-coupling}
\textcolor{black}{We consider a }special type of coupling $(Q_1, Q_2)$ such that the projection maps from $Z_{12}$ to $Z_2$ are two-fold.
\begin{definition}
A coupling of orthodiagonal quadrilaterals $(Q_1, Q_2)$ is called {\it involutive}, if $Z_{12}$ admits an involution $j_{12}$ that preserves $x_1, x_3$ but changes~$x_2$ almost everywhere,
$$
j_{12}:Z_{12} \to Z_{12},\quad j_{12}(x_1, x_2, x_3) = (x_1, x_2', x_3).
$$
\end{definition}
\begin{lemma}\label{lma4.2}
A coupling of orthodiagonal quadrilaterals $(Q_1, Q_2)$ is involutive if and only if parameters $F_2, \lambda_1, \lambda_2$ satisfy:
    \begin{equation}\label{eq10}
        \begin{tabular}{ |c|c|c|c| } 
            \hline
              $F_2$ & $0$ & $\mathbb{R}\backslash 0$ & $\pm \infty$ \\ 
            \hline
               & $\lambda_1 = \lambda_2$ & $\lambda_1 = \lambda_2 = -1$ &$\lambda_1 = \lambda_2^{-1}$\\
            \hline
        \end{tabular}
    \end{equation}
\end{lemma}
\begin{proof}
A coupling $(Q_1, Q_2)$ is involutive iff involutions $j_1, j_2$ of $Z_1, Z_2$ satisfy $j_1(x_2) = j_2(x_2)$.

\textcolor{black}{The involution $j_2$ for $Z_2$ can be induced from $\hat{j}_{2}:\widehat{Z}_2 \to \widehat{Z}_2$, $\hat{j}_2(y_2) = \lambda_2 y^{-1}_2$ using the~relation~$y_2(x_2)$ from~\cref{eq8}. Denoting $x'_2 = j_2(x_2)$, $y'_2 = \hat{j}_2(y_2)$ and solving the following equation for $x'_2$, we get
\begin{equation*}
    \frac{x'_2 + F_2}{1 - F_2 x'_2} = \lambda_2\frac{1 - F_2 x_2}{x_2 + F_2}\, \Rightarrow \,
    j_2(x_2) = \frac{-(\lambda_2 + 1) F_2 x_2 + \lambda_2 - F_2^2}{(1 - \lambda_2 F_2^2) x_2 + (\lambda_2 + 1) F_2}.
\end{equation*}}

Therefore, $j_1(x_2) = j_2(x_2)$ if and only if the following equation,
\begin{equation}\label{eq11}
\frac{\lambda_1}{x_2} - \frac{-(\lambda_2 + 1) F_2 x_2 + \lambda_2 - F_2^2}{(1 - \lambda_2 F_2^2) x_2 + (\lambda_2 + 1) F_2} = 0,
\end{equation}
is true for any $x_2$. It is equivalent to the conditions \cref{eq10}.
\end{proof}
\begin{lemma}\label{lma3.4}
Let $(Q_1, Q_2)$ be an involutive coupling of orthodiagonal quadrilaterals. The quotient space $W = Z_{12} / j_{12}$ has the following form
    \begin{itemize}
        \item If $Q_1$ and $Q_2$ are both elliptic, then
            \begin{equation}\label{eq12}
                 k_2 \frac{4F_2 x_1^2 + \nu_1(1 - F_2^2)x_1 + 4F_2\mu_1}{(1 - F_2^2)x_1^2 - F_2\nu_1x_1 + \mu_1(1 - F_2^2)}\Bigg(\frac{x_3 + F_3}{1 - F_3 x_3} + \mu_2 \frac{1 - F_3 x_3}{x_3 + F_3}\Bigg) = \nu_2.
            \end{equation}
        \item If $Q_2$ is an (anti)deltoid laterally coupled to $Q_1$ that is elliptic, then 
            \begin{equation}\label{eq13}
                k_2 \frac{4F_2 x_1^2 + \nu_1(1 - F_2^2)x_1 + 4F_2\mu_1}{(1 - F_2^2)x_1^2 - F_2\nu_1x_1 + \mu_1(1 - F_2^2)} = \xi_2 \Bigg(\frac{x_3 + F_3}{1 - F_3 x_3}\Bigg)^{n_2}.
            \end{equation}
          \item If $Q_1,Q_2$ are laterally coupled (anti)deltoids, then 
            \begin{equation}\label{laterally}
                k_2 \frac{(1-F_2^2)\xi_1 x_1^{n_1} + 4F_2}{-F_2 \xi_1 x_1^{n_1} + (1 - F_2^2)} = \xi_2 \Bigg(\frac{x_3 + F_3}{1 - F_3 x_3}\Bigg)^{n_2}.
            \end{equation}
    \end{itemize}
    Here $\mu_i, \nu_i, \xi_i, n_i$ are as in~\cref{lma2.7} and $k_2 = -\lambda_2$ if $F_2= \pm \infty$, or $k_2 = 1$ otherwise.
\end{lemma}
\begin{proof}

    We first consider the case when $Q_1, Q_2$ are both elliptic. The space $Z_{12}$ is described by the system
    $$
        \begin{cases}
        (x_2 + \frac{\lambda_1}{x_2})(x_1 + \frac{\mu_1}{x_1}) = \nu_1, \\
        (y_2 + \frac{\lambda_2}{y_2})(y_3 + \frac{\mu_2}{y_3}) = \nu_2.
        \end{cases}
    $$
    First, using \cref{eq8} and \cref{eq10} we rewrite the part involving $y_2$ in the second equation as follows:
    \begin{equation*}
      y_2 + \frac{\lambda_2}{y_2} = \frac{(1 + \lambda_2 F_2^2)x_2^2 + 2F_2(1 - \lambda_2)x_2 + (F_2^2 + \lambda_2)}{-F_2x_2^2 + (1 - F_2^2) x_2 + F_2} = \begin{cases}
          \mfrac{x_2^2 + \lambda_2}{x_2} = x_2 + \mfrac{\lambda_1}{x_2}, \quad \text{if } F_2 = 0, \\[2ex]
          \mfrac{(1 - F_2^2)x_2^2 + 4 F_2 x_2 + (F_2^2 - 1)}{-F_2x_2^2 + (1 - F_2^2) x_2 + F_2} = \mfrac{(1 - F_2^2)(x_2 - x_2^{-1}) + 4F_2}{-F_2(x_2 - x_2^{-1}) + (1 - F_2^2)}, \quad \text{if } F_2 \in \mathbb{R}\backslash 0,\\[2ex]
        \mfrac{\lambda_2 x_2^2 + 1}{-x_2} = -\lambda_2 (x_2 + \mfrac{\lambda_1}{x_2}), \quad \text{if } F_2 = \pm \infty.
      \end{cases}
    \end{equation*}
    Then we rewrite these three cases as a function of $w_2 = x_2 + \mfrac{\lambda_1}{x_2}$ and obtain the following equality:
    $$
    y_2 + \frac{\lambda_2}{y_2} = k_2 \frac{(1-F_2^2)w_2 + 4F_2}{-F_2 w_2 + (1 - F_2^2)},
    $$
    where $k_2 = -\lambda_2$ if $F_2= \pm \infty$, or $k_2 = 1$ otherwise.
    
    \textcolor{black}{After expressing $w_2$ as $\frac{\nu_1 x_1}{x_1^2 + \mu_1}$ from the first equation we substitute it in the second equation of the system using the above equality.  Using the expression \cref{eq8} for $y_3$ we obtain the desired result. The remaining cases (\ref{eq13}) and (\ref{laterally}) are treated in an analogous way.}

\end{proof}
\begin{definition}\label{def3.6}
    Orthodiagonal quadrilaterals $Q_1$ and $Q_2$ are called {\it compatible} if one of the~following holds:
    \begin{itemize}
        \item $Q_1$ and $Q_2$ are involutive;
        \item $Q_1$ and $Q_2$ are either a deltoid or antideltoid and they are frontally coupled.
    \end{itemize}
\end{definition}

\subsection{Matching of two involutive couplings}\label{matching-involutive}

In order to compose two couplings to a mechanism,
it is sufficient to suppose that $(Q_1, Q_2)$, $(Q_3, Q_4)$ are orthodiagonal involutive couplings and that the algebraic sets $Z_{12}/ j_{12}$, $Z_{34}/ j_{34}$ are identical. In this case, we define the class from the assumption that 
\begin{equation}\label{eq14}
Z_{12} / j_{12} = W = Z_{34}/ j_{34},
\end{equation}
where $W = \{(x_1, x_3) \in (\mathbb{C}P^1)^2|\,\exists x_2, x_4 \in \mathbb{C}P^1\text{ s. t. } (x_1, x_2, x_3) \in Z_{12}, (x_1, x_4, x_3) \in Z_{34}\}$.

We determine the conditions when \cref{eq14} holds and say such combination $(Q_1,Q_2,Q_3,Q_4)$ is of {\it orthodiagonal involutive (OI)} type. 

First, we notice the following property of involutive couplings.
\begin{lemma}\label{lma4.1}
    If $(Q_1, Q_2)$, $(Q_3, Q_4)$ are OI couplings and \cref{eq14} is satisfied, then $(Q_1, Q_4)$ and $(Q_2, Q_3)$ are compatible.
\end{lemma} 
\begin{proof}
    We start with the case when the equation of $W$ is of the form \cref{eq12} (where $k_4 = -\lambda_4$ if $F_4 = \pm \infty$ or $k_4 = 1$ otherwise)
   \begin{equation}\label{eq15}
   k_2\frac{4F_2 x_1^2 + \nu_1(1 - F_2^2)x_1 + 4F_2\mu_1}{(1 - F_2^2)x_1^2 -F_2\nu_1x_1 + \mu_1(1 - F_2^2)} \Big(y_3 + \frac{\mu_2}{y_3}\Big) = \nu_2,
   \end{equation}
    \begin{equation*}
        k_4\frac{4F_4 x_3^2 + \nu_3(1 - F_4^2)x_3 + 4F_4\mu_3}{(1 - F_4^2)x_3^2 - F_4\nu_3x_3 + \mu_3(1 - F_4^2)}\Big(y_1 + \frac{\mu_4}{y_1}\Big) = \nu_4.
   \end{equation*}
    \textcolor{black}{Then the involutions for $x_1, y_3$ and $x_3, y_1$ are defined as follows:
    \begin{equation*}
       i_1(x_1) = \frac{\mu_1}{x_1}, \quad \hat{i}_2(y_3) = \frac{\mu_2}{y_3} \quad \text{and} \quad  i_3(x_3) = \frac{\mu_3}{x_3}, \quad \hat{i}_4(y_1) = \frac{\mu_4}{y_1},
    \end{equation*}
    which can be checked by direct substitution to the equations.}
    
   Since $Z_{12} / j_{12}, Z_{34} / j_{34}$ are defined by equations \cref{eq15} respectively, the involutions for both equations must coincide (see \cref{eq11}). For the equations of the form \cref{eq13} or \cref{laterally} the reasoning is similar. We only need to notice that if $Q_2$ is an (anti)deltoid, there will be no involution for $y_3$ (or $x_3$),
   and thus  $Q_3$ must be an (anti)deltoid and hence $(Q_2,Q_3)$ have to be frontally coupled.
\end{proof}
\begin{remark}
    Note that even if spaces $Z_{12}$ and $Z_{34}$ do not coincide it is still possible for them to have a common component if they are reducible. Such a case leads to subclasses of a different type of Kokotsakis meshes, called {\it linear compound} (Section 3.5.2 and 3.5.4 in \cite{izmestiev-2017}). These subclasses have a more simple relation between $x_1, y_3$ (or $x_1, x_3$), i.e., the numerator in the relation~\cref{eq15} has a total degree of two rather than four. Therefore, we omit this case.
\end{remark}
\begin{remark}
In this paper, we mainly focus on the rich class of flexible \textcolor{black}{skew Kokotsakis meshes} composed of orthodiagonal elliptic quadrilaterals. For complexes with two (anti)deltoids we refer to Appendix B. There we present the conditions and an example of a mechanism with two involutive couplings which contain a deltoid and antideltoid respectively. Note that this is not possible in the case of planar
quads. We skip the case when all quadrilaterals are deltoids or antideltoids, as it is similar to the linear compounds mentioned above.
\end{remark}

\section{Elliptic orthodiagonal involutive type}\label{sec5}

\textcolor{black}{In this section, we present algebraic conditions for a flexible skew Kokotsakis mesh which belongs to the {\it elliptic orthodiagonal involutive} type. It is the most generic type for a mechanism constructed from orthodiagonal elliptic quadrilaterals with involutive coupling. Moreover, in Section~\ref{sec5.5}, we will prove the theorem that ensures the existence of a real interval of solutions.}

\textcolor{black}{Firstly, we justify our definition of orthodiagonal involutive type with elliptic quadrilaterals by proving the following corollary.}
\begin{corollary}\label{coro}
    If $(Q_1, Q_2)$, $(Q_3, Q_4)$ are elliptic OI couplings and \cref{eq14} is satisfied, then $(Q_1, Q_4)$ and $(Q_2, Q_3)$ are also involutive and hence the following table of conditions holds.
        \begin{equation}\label{elp-table}
        \begin{tabular}{ |c|c|c|c| } 
            \hline
              $F_i$ & $0$ & $\mathbb{R}\backslash 0$ & $\pm \infty$ \\ 
            \hline
             i=2,4  & $\lambda_{i-1} = \lambda_{i}$ & $\lambda_{i - 1} = \lambda_{i} = -1$ &$\lambda_{i-1} = \frac{1}{\lambda_{i}}$\\
            \hline
              $F_j$ & $0$ & $\mathbb{R}\backslash 0$ & $\pm \infty$ \\ 
            \hline
             j=1,3  & $\mu_{j-1} = \mu_j$ & $\mu_{j-1} = \mu_j = -1$ &$\mu_{j-1} = \frac{1}{\mu_j}$\\
            \hline
        \end{tabular}
    \end{equation}
\end{corollary} 

\textcolor{black}{It is an immediate implication of Lemma \ref{lma4.1}.} Note that equation \cref{eq14} only guarantees that two combined couplings $(Q_1, Q_2)$ and $(Q_3, Q_4)$ have a one-parameter solution of $x_i, y_i$ in $\mathbb{C}$. However, to obtain a flexible Kokotsakis mesh, we need a one-parameter set of real solutions. This problem will be resolved in Theorem~\ref{thm:main}.


Starting from equation \cref{eq15}, we consider the following substitutions (note that $t_n=\infty$ if $F_n^2=1$)
\begin{equation*}
w_i=x_i+\frac{\mu_{i}}{x_i}, i=1,3; \quad w_j=x_j+\frac{\lambda_{j-1}}{x_j}, j=2,4; \quad t_n = \frac{2 F_n}{1 - F_n^2}, 1\leq n \leq 4.
\end{equation*}

Since~$(Q_i, Q_{i + 1})$ are involutive, we can rewrite the factors in the left-hand side of \cref{eq15} as follows
\begin{equation}\label{eq28}
     k_2 \frac{4F_2 x_1^2 + \nu_1(1 - F_2^2)x_1 + 4F_2\mu_1}{(1 - F_2^2)x_1^2 - F_2\nu_1x_1 + \mu_1(1 - F_2^2)} = k_2 \frac{4 t_2 w_1 + 2 \nu_1}{2 w_1 - \nu_1 t_2}\quad \text{and} \quad y_3 + \frac{\mu_2}{y_3} = k_3 \frac{2 w_3 + 4 t_3}{2 - t_3 w_3}.
\end{equation}

\begin{remark}\label{rm4.2}
    A similar multiplier $k_4$ also appears in the equation for $(Q_3, Q_4)$. From now on we suppose for convenience that $F_i \neq \pm \infty$. However, the  following computations do not change drastically if
    $F_i$ takes values $\pm \infty$. 
\end{remark}

\textcolor{black}{Then \cref{eq15} can be rewritten as follows:}
\begin{equation}\label{eq31}
    \frac{4 t_2 w_1 + 2\nu_1}{2 w_1 - \nu_1 t_2}\cdot\frac{2 w_3 + 4 t_3}{2 - t_3 w_3} = \nu_2  \quad \text{and} \quad \frac{4 t_4 w_3 + 2\nu_3}{2 w_3 - \nu_3 t_4}\cdot\frac{2 w_1 + 4 t_1}{2 - t_1 w_1}= \nu_4,
\end{equation}
or
\begin{equation}\label{eq16}
    \begin{gathered}
        2(4t_2 + \nu_2t_3)w_1 w_3 + 4(4t_2 t_3 - \nu_2) w_1+\nu_1 (4 - \nu_2 t_2 t_3)w_3 + 2 \nu_1 (\nu_2 t_2 + 4 t_3) = 0,\\
        2(\nu_4 t_1 + 4 t_4) w_1 w_3 + \nu_3 (4 - \nu_4 t_1 t_4) w_1 + 4 (4 t_1 t_4 - \nu_4) w_3 + 2 \nu_3 (4 t_1 + \nu_4 t_4) = 0.
    \end{gathered}
\end{equation}
\begin{remark}\label{rm4.3}
    If some $F_i^2 = 1$ (i.e. $t_i=\infty$), equations \cref{eq31} and \cref{eq16} can still be used. For example, if $F_2 = 1$, we have
        \begin{equation*}
            -\frac{4 w_1}{\nu_1}\cdot\frac{2 w_3 + 4 t_3}{2 - t_3 w_3} = \nu_2.
        \end{equation*}
\end{remark}
According to $Z_{12} / j_{12} =  Z_{34}/ j_{34}$, the equations in \cref{eq16} define the same space $W$, so the coefficients of corresponding terms are proportional. This is equivalent to the condition that $\rk(N) = 1$, where $N$ is the following matrix
\begin{equation}\label{eq17}
N =
    \begin{pmatrix}
        2(4 t_2 + \nu_2 t_3) & 4(4 t_2 t_3 - \nu_2) & \nu_1(4 - \nu_2 t_2 t_3) &     2\nu_1(\nu_2 t_2 + 4 t_3)\\
        2(\nu_4 t_1 + 4 t_4) & \nu_3(4 - \nu_4 t_1 t_4) & 4(4t_1 t_4 - \nu_4) & 2 \nu_3(4 t_1 + \nu_4 t_4)
    \end{pmatrix}.
\end{equation}
Therefore, it is equivalent to a zero-determinant condition for all 2-minors of the matrix $N$, where $p_{ij}$ stands for the minor of  columns $i, j$.
\begin{equation}\label{eq18}
    \begin{gathered}
        p_{12} = 16\nu_4t_1t_2t_3 + 4\nu_3\nu_4t_1t_2t_4 + \nu_2\nu_3\nu_4t_1t_3t_4 + 64t_2t_3t_4 - 4\nu_2\nu_4t_1 - 16\nu_3t_2 - 4\nu_2\nu_3t_3 - 16\nu_2t_4,\\
        p_{13} = \nu_1\nu_2\nu_4t_1t_2t_3 + 64t_1t_2t_4 + 16\nu_2t_1t_3t_4 + 4\nu_1\nu_2t_2t_3t_4 - 4\nu_1\nu_4t_1 - 16\nu_4t_2 - 4\nu_2\nu_4t_3 - 16\nu_1t_4,\\
        p_{14} = (16 \nu_3 - \nu_1  \nu_2 \nu_4) t_1 t_2 + 4 (\nu_3 \nu_2 - \nu_1 \nu_4) t_1 t_3 + 4(\nu_3 \nu_4 - \nu_1 \nu_2) t_2 t_4 + ( \nu_3 \nu_4 \nu_2 - 16 \nu_1) t_3 t_4,\\
        p_{23} = (256 - \nu_1 \nu_2 \nu_3 \nu_4) t_1 t_2 t_3 t_4 + 4(\nu_1 \nu_3  \nu_4 - 16 \nu_2) t_1 t_4 +4(\nu_1  \nu_2 \nu_3 - 16 \nu_4) t_2 t_3 + 16(\nu_4 \nu_2 - \nu_1 \nu_3),\\
        p_{24} = \nu_3(64 t_1 t_2 t_3 + \nu_1\nu_2\nu_4t_1t_2t_4 + 4\nu_1\nu_4t_1t_3t_4 + 16\nu_4t_2t_3t_4 - 16\nu_2t_1 - 4\nu_1\nu_2t_2 - 16\nu_1t_3 - 4\nu_2\nu_4t_4),\\
        p_{34} = \nu_1(4\nu_2\nu_3t_1t_2t_3 + 16\nu_2t_1t_2t_4 + 64t_1t_3t_4 + \nu_2\nu_3\nu_4t_2t_3t_4 - 16\nu_3t_1 - 4\nu_2\nu_4t_2 - 16\nu_4t_3 - 4\nu_3\nu_4t_4).\\
    \end{gathered}
\end{equation}

\noindent \emph{Pseudo-planar case.}
\textcolor{black}{In the pseudo-planar case (see \cref{def-pseudoplanar}) the polynomial system \cref{eq18}} degenerates to $\nu_2 \nu_4 - \nu_1 \nu_3 = 0$ and equality of involution factors from table \cref{elp-table} implies that the parameters $(\lambda_1, \lambda_3, \mu_1, \mu_2)$ and $(\nu_1, \nu_2, \nu_3)$ are independent and they determine a finite family of flexible \textcolor{black}{skew Kokotsakis meshes} uniquely. In geometric terms, these parameters can be described as seven planar angles $(\delta_1,\delta_2, \delta_3, \alpha_1, \alpha_3, \gamma_1, \gamma_2)$ and one dihedral angle~$\tau_1$, which gives us 8 degrees of freedom.

\subsection{Existence}\label{sec5.5}
    
We define the subclass of the elliptic OI type as follows.   
\begin{enumerate}
    \item All planar angles $(\alpha_i , \beta_i, \gamma_i, \delta_i)$ satisfy the conditions of orthodiagonality and ellipticity:
    $$
    \begin{matrix}
    \cos(\alpha_i) \cos(\gamma_i) = \cos(\beta_i) \cos(\delta_i), & \alpha_i \pm \beta_i \pm \gamma_i \pm \delta_i \neq 0 \text{ mod } 2\pi.
    \end{matrix}
    $$
    \item Elliptic quadrilaterals $Q_i$ satisfy involutive coupling conditions from table \cref{elp-table}.
    \item The set of parameters $(\nu_1, \nu_2, \nu_3, \nu_4, t_1, t_2, t_3, t_4)$ satisfies the polynomial system \cref{eq18}.
    \end{enumerate}
\textcolor{black}{
The flexibility of the \textcolor{black}{skew Kokotsakis mesh} means the existence of a one-parametric family of real solutions (not a discrete set of points) for the system \cref{eq3}. Firstly we give necessary and sufficient conditions when a single equation of the system has a one-parameter set of real solutions.
 \begin{proposition}
     \label{thm:localexistence} (Local existence) 
     Equation \eqref{eq5} which defines the configuration space of an orthodiagonal quadrilateral $Q_i$ of elliptic type has solutions in real if and only if one of the following conditions is satisfied:
     \begin{itemize}
     \item at least one value of $\lambda_i$ or $\mu_i$ is negative;
     \item $\lambda_i > 0, \,\mu_i > 0$ and $\mfrac{\nu_i^2}{\lambda_i \mu_i} > 16$.
     \end{itemize}
 \end{proposition}}
 \textcolor{black}{\begin{proof}
   Suppose that $i = 1$ and $\lambda_1, \mu_1$ are positive, then the equation \cref{eq5} for a spherical quad $Q_1$ can be rewritten in the following way:
     \begin{equation*}
        \Big(\frac{x_2}{\sqrt{\lambda_1}} + \frac{\sqrt{\lambda_1}} {x_2}\Big) \Big(\frac{x_1}{\sqrt{\mu_1}} + \frac{\sqrt{\mu_1}} {x_1}\Big) = \frac{\nu_1}{\sqrt{\lambda_1 \mu_1}}.
    \end{equation*}
    Since absolute values of both parentheses are greater or equal to 2, such an equation has real solutions if and only if $\mfrac{\nu_1^2}{\lambda_1 \mu_1} > 16$.
    For negative values of $\lambda_i$ or $\mu_i$, the proposition is trivial since the discriminant of a corresponding quadratic equation is always positive.
 \end{proof}}
\textcolor{black}{
Even when every single equation has an interval of a real solution, the intervals still may not intersect. The following main result provides a constructive answer to this problem, expressed in terms of inequalities for involutive factors.}

 \begin{theorem} \label{thm:main} \textcolor{black}{(Global existence) A skew Kokotsakis mesh that belongs to the elliptic OI type and satisfies conditions for local existence is flexible if and only if }
        \begin{itemize}

            \item $\lambda_i < 0$ and $\mu_i > 0$ for all $i$ and
\begin{equation}\label{global-existence-condition}
                t_2 = 0 \quad \text{or} \quad 16\sqrt{\mu_1\mu_2}-|\nu_1\nu_2| < \frac{4|\nu_2|\sqrt{\mu_1}+4|\nu_1|\sqrt{\mu_2}}{|t_2|},
            \end{equation}

            \item $\lambda_i > 0$ and $\mu_i < 0$ for all $i$ and

   \begin{equation}\label{global-existence-condition1}
            t_3 = 0 \quad \text{or} \quad 16\sqrt{\lambda_2\lambda_3}-|\nu_2\nu_3| < \frac{4|\nu_3|\sqrt{\lambda_2}+4|\nu_2|\sqrt{\lambda_3}}{|t_3|},
        \end{equation}

        \item all other combinations of signs for $\lambda_i, \mu_i$.
        \end{itemize}
            
        

    \end{theorem}

\textcolor{black}{
The proof is based on the explicit description of intervals of solutions for every single equation and finding conditions when these intervals intersect. Technical details can be found in Appendix~\ref{apx:a}.
}

\subsection{Geometric properties}\label{sec6.2}    
    
Generalized Kokotsakis meshes of elliptic orthodiagonal involutive type have some remarkable properties which are analogous to the ones in the planar case.
\begin{enumerate}
    \item The conditions $\cos(\alpha_i) \cos(\gamma_i) = \cos(\beta_i) \cos(\delta_i)$ means that the plane $[B_i, A_i, A_{i + 1}]$ is orthogonal to the plane $[C_i, A_i, A_{i - 1}]$ for $i = 1 \ldots 4$ (see \cref{fig1}), and this property is preserved during the flexion of the \textcolor{black}{mesh}.
    \item The conditions on involution factors in each subclass means that an angle between planes $[B_i, A_i, A_{i + 1}]$ and $[A_i, A_{i + 1}, C_{i + 1}]$ is equal to $\zeta_{i + 1}$ for $i = 1 \ldots 4$, and this property is preserved during a flexion of the \textcolor{black}{mesh}.
\end{enumerate}

The involutive coupling of quadrilaterals implies compatibility of the involutions $j_1$ and $j_2$ (or the condition \cref{eq10} on the involution factors) for the coupling $(Q_1, Q_2)$. In geometrical terms it can be interpreted as conditions on the angles $\xi_1, \xi_2, \xi_1', \xi_2'$, which are expressed in the following result.

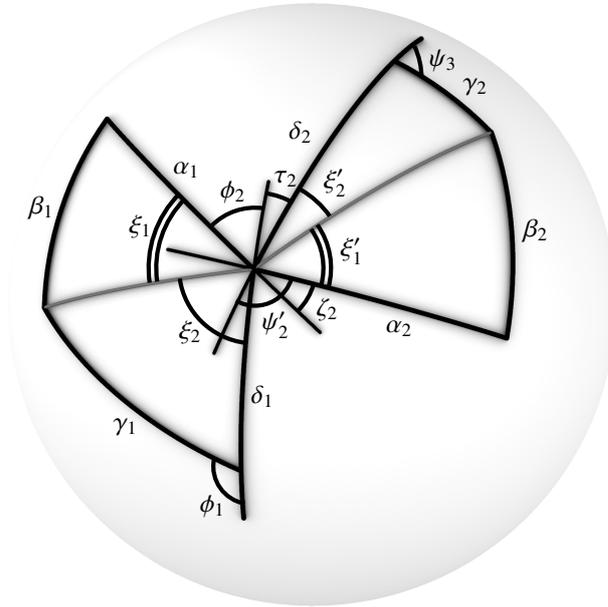
\begin{figure}[ht!]
    \centering
    \begin{tikzpicture}
        \node[anchor=south west,inner sep=0] at (-3.5, -4.7) {\includegraphics[width=0.5\textwidth]{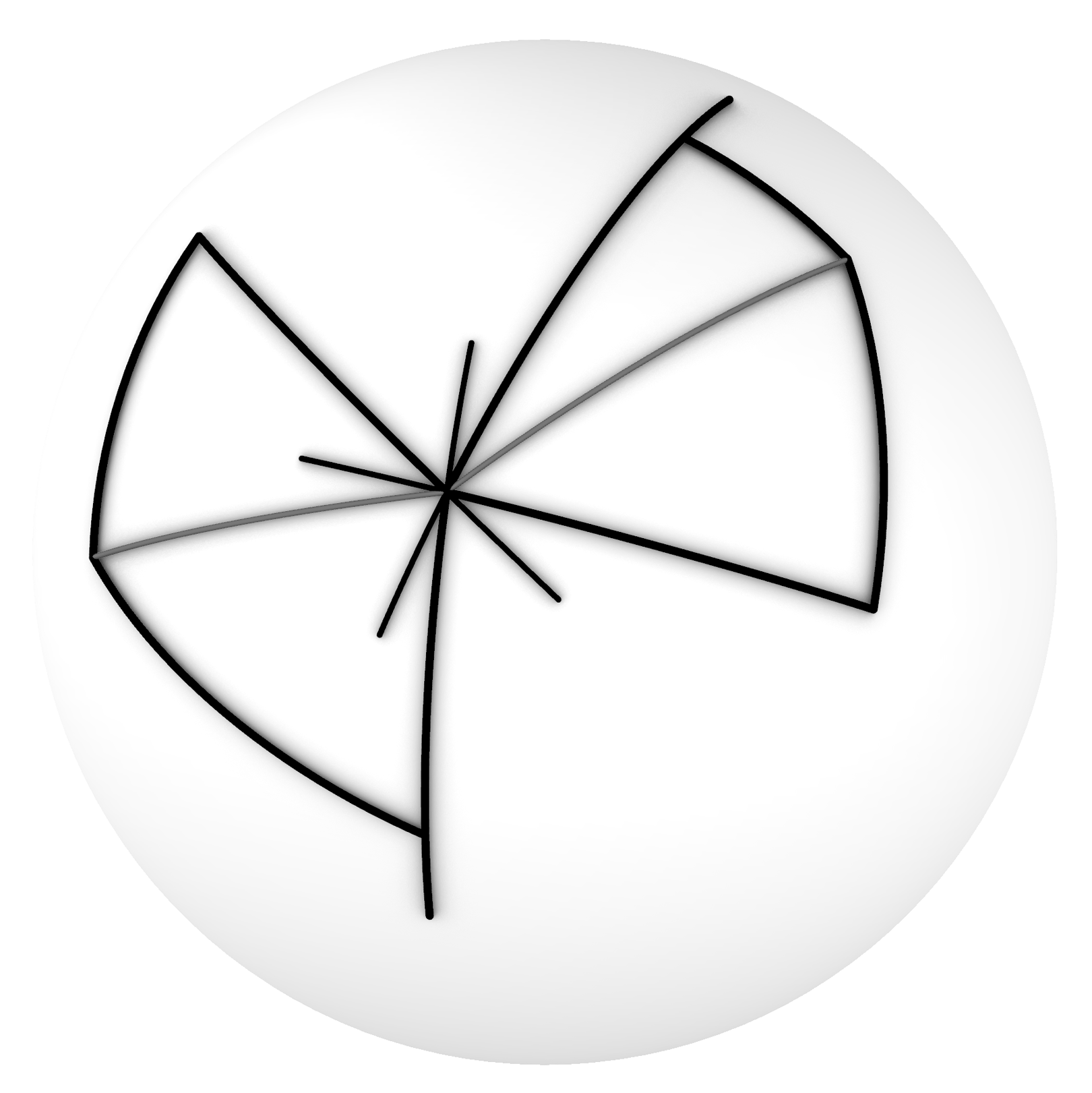}};

        \draw[ultra thick] (240:0.5) arc (240:345:0.5);
        \draw (290:0.8) node {$\psi'_2$};
        
        \draw[ultra thick] (315:0.8) arc (315:345:0.8);
        \draw (330:1.1) node {$\zeta_2$};
        
        \draw[ultra thick] (-15:0.9) arc (-15:35:0.9);
        \draw[ultra thick] (-15:1) arc (-15:35:1);
        \draw (10:1.3) node {$\xi'_1$};

        \draw[ultra thick] (35:1.2) arc (35:60:1.2);
        \draw (47:1.55) node {$\xi'_2$};

        \draw[ultra thick] (60:1) arc (60:82:1);
        \draw (71:1.25) node {$\tau_2$};
        
        \draw[ultra thick] (82:0.8) arc (82:135:0.8);
        \draw (106:1.05) node {$\phi_2$};

        \draw[ultra thick] (135:1.3) arc (135:188:1.3);
        \draw[ultra thick] (135:1.4) arc (135:188:1.4);
        \draw (160:1.6) node {$\xi_1$};
        
        \draw[ultra thick] (188:1) arc (188:263:1);
        \draw (225:1.2) node {$\xi_2$};
        
        \draw[ultra thick] (258:2.6) arc (170:260:0.5);
        \draw (260:3.2) node {$\phi_1$};
        
        \draw[ultra thick] (49:3.34) arc (0:25:1);
        \draw (48:3.7) node {$\psi_3$};

        \draw (-0.9, 1.3) node {$\alpha_1$};
        \draw (-2.8, 0.8) node {$\beta_1$};
        \draw (-1.7, -2.1) node {$\gamma_1$};
        \draw (0.1, -1.7) node {$\delta_1$};

        \draw (1.9, -0.8) node {$\alpha_2$};
        \draw (3.7, 0.5) node {$\beta_2$};
        \draw (2.9, 2.4) node {$\gamma_2$};
        \draw (0.6, 1.8) node {$\delta_2$};

    \end{tikzpicture}
    \caption{Illustration of Lemma~\ref{lem-final}.}
    \label{fig7}
\end{figure}

\begin{lemma} \label{lem-final}
Let the spherical elliptic orthodiagonal quadrilaterals $Q_1$ and $Q_2$ form an involutive coupling as in \cref{fig7}. Then
\textcolor{black}{$$
\begin{matrix}
\xi_1 = \xi_1', &
\xi_2 = \xi_2' + \tau_2 + \zeta_2
\end{matrix}
$$}
during flexions.
\end{lemma}
\begin{proof}
From the equality $\psi_2 = \phi_2 + \tau_2 + \zeta_2$ for the linkage $(Q_1, Q_2)$ we obtain $\pi - (\xi_1' + \xi_2') = \pi - (\xi_1 + \xi_2) + \tau_2 + \zeta_2$. The condition \cref{eq10} means compatibility of the involutions $j_1$ and $j_2$, i.e $j_1(x_2) = j_2(x_2)$. Geometrically, it means that after folding along each diagonal in $Q_1$ and $Q_2$ in \cref{fig7} we obtain the angles $\phi'_2, \psi'_2$ from $\phi_2, \psi_2$, respectively, such that $\psi'_2 = \phi'_2 + \tau_2 + \zeta_2$. The last condition is equivalent to
\textcolor{black}{\begin{equation*}
\pi - (\xi_2' - \xi_1') = \pi - (\xi_2 - \xi_1) + \tau_2 + \zeta_2.
\end{equation*}}
\end{proof}

\section{Practical construction \textcolor{black}{ and fabrication}} \label{sec:practice}
In comparison to the planar orthodiagonal involutive case, in the skew one, the system of involutive factors should also satisfy conditions \cref{eq18}. In this section, we show how to obtain all solutions to this system. Our approach is to use a powerful tool from computer algebra: a cylindrical algebraic decomposition.

\subsection{Polynomial system}\label{sec7.1}

In this section we will show how to construct all solutions to the polynomial system \cref{eq18}. Let us notice that the elements of the first column are both zero or non-zero
\begin{equation}
    \begin{pmatrix}
        2(4 t_2 + \nu_2 t_3) \\
        2(\nu_4 t_1 + 4 t_4)
    \end{pmatrix},
\end{equation}
because rows are proportional in matrix \cref{eq17}.

    Firstly, we consider the generic case when
    \begin{equation}\label{eq26}
        t_2 \neq -\frac{\nu_2}{4} t_3, \quad t_4 \neq -\frac{\nu_4}{4} t_1,
    \end{equation} 
    and under such restrictions the system \cref{eq18} is equivalent to the system of $p_{12}, p_{13}, p_{14}$. Indeed, because each $p_{ij}$ is equal to zero if and only if the $i$-th and $j$-th columns are linearly dependent, linear dependence of the first column with second and third columns leads to the linear dependence between second and third column. Therefore, we describe the set of solutions for the first three polynomials. Since $p_{12}, p_{13}$ are linear in $t_4$, they have a common
    root if 
    \vspace{-0.25cm}
        \begin{multline}\label{eq19}
            \nu_4(\nu_1 \nu_2 \nu_3 \nu_4 - 256) t_1^2 t_2 t_3 + 4 \nu_4(16 \nu_2 - \nu_1 \nu_3 \nu_4) t_1^2 + 16 \nu_3 (16 - \nu_4^2) t_1 t_2 + \\ +4 \nu_2 \nu_3 (16 - \nu_4^2) t_1 t_3 + 16 (\nu_1 \nu_2 \nu_3 - 16 \nu_4) t_2 t_3 + 64 (\nu_2 \nu_4 - \nu_1 \nu_3) = 0.
        \end{multline}
        Since the equation is linear in $t_2$, we can explicitly express $t_2$ with the other variables
        \begin{equation}\label{eq23}
            t_2 = 4\frac{\nu_4(\nu_1 \nu_3 \nu_4 - 16 \nu_2)t_1^2 + \nu_2 \nu_3(\nu_4^2 - 16)t_1 t_3 + 16(\nu_1 \nu_3 - \nu_2 \nu_4)}{\nu_4(\nu_1 \nu_2 \nu_3 \nu_4 - 256)t_1^2 t_3 + 16 \nu_3 (16 - \nu_4^2) t_1 + 16(\nu_1 \nu_2 \nu_3 - 16 \nu_4)t_3},
        \end{equation}
        and from $p_{12}$ we can determine $t_4$,
            \begin{equation}\label{denom30}
                t_4 = 4\frac{(-16 \nu_1 \nu_4+\nu_2^2 \nu_4 \nu_1) t_1 t_3+(-16 \nu_2 \nu_4+\nu_1 \nu_2^2 \nu_3) t_3^2-16 \nu_2 \nu_4+16 \nu_1 \nu_3}{(\nu_2^2 \nu_4 \nu_1 \nu_3-256 \nu_2) t_1 t_3^2+(16 \nu_1 \nu_3 \nu_4-256 \nu_2) t_1+(-16 \nu_1 \nu_2^2+256 \nu_1) t_3}.
            \end{equation}
        Using these substitutions for the last polynomial $p_{14}$ we have the following form
        \begin{equation}\label{eq20}
         (a_{22} t_1^2 t_3^2 + a_{20} t_1^2 + a_{02} t_3^2 + a_{00})(b_{22} t_1^2 t_3^2 + b_{20} t_1^2 + b_{02}t_3^2 + b_{00})=0,
        \end{equation}
        where
        \begin{align*}
            a_{22} = (\nu_1 \nu_2 \nu_3 \nu_4-256) (\nu_1 \nu_4-\nu_2 \nu_3), && b_{22} = \nu_2 \nu_4 (\nu_1 \nu_2 \nu_3 \nu_4-256),\\
            a_{20} = (\nu_1 \nu_3 \nu_4-16 \nu_2) (\nu_1 \nu_2 \nu_4-16 \nu_3), && b_{20} = 16 \nu_4 (\nu_1 \nu_3 \nu_4-16 \nu_2),\\
            a_{02} = (16 \nu_1 - \nu_2 \nu_3 \nu_4) (\nu_1 \nu_2 \nu_3 - 16 \nu_4), && b_{02} = 16 \nu_2 (\nu_1 \nu_2 \nu_3 - 16 \nu_4),\\
            a_{00} = 16 (\nu_1 \nu_3-\nu_2 \nu_4) (\nu_1 \nu_2-\nu_3 \nu_4), && b_{00} =  256 (\nu_1 \nu_3 - \nu_2 \nu_4).
        \end{align*}
        We have two families of solutions. The left polynomial has a one-dimensional family of solutions for variables $t_1, t_3$,
        from which we can determine $t_2,t_4$ via (\ref{eq23}) and
        (\ref{denom30}), as long as denominators in \cref{eq23}
        do not vanish.

\begin{remark}
We see that in the general case, the solution set can be described by five variables $(t_1, \nu_1,..,\nu_4)$. From condition \cref{eq10} we deduce that if $t_i \neq 0$ for all $i$ the
Kokotsakis mesh is geometrically determined by the four planar angles $(\delta_1,..,\delta_4)$ and two dihedral angles~$(\tau_1, \zeta_1)$, which gives us 6 degrees of freedom. It is remarkable that in contrast to the planar case of OI type, all four angles in the central quad can be chosen arbitrarily.
\end{remark}
        
        For a solution $(t_1, t_3)$ of the right polynomial in (\ref{eq20}), the further variables $t_2, t_4$ are obtained via
        \begin{equation}\label{eq27}
            t_2 = -\frac{\nu_2}{4} t_3, \quad t_4 = -\frac{\nu_4}{4} t_1.
        \end{equation} 
        Indeed, if we solve the second polynomial for $t_1^2$ and substitute this parametrization to \cref{eq23}, we obtain the above result. However, it contradicts our assumption \cref{eq26}.

        Now we suppose that \cref{eq27} is satisfied. Then the second column of $N$ is defined by non-zero elements $-4 \nu_2 (t_3^2 + 1) \ne 0,$ and $-4 \nu_4 (t_1^2 + 1) \ne 0$, since $\nu_i$ cannot be zero. So \cref{eq18} is equivalent to the system of $p_{23} = 0, p_{24} = 0$, which has the following form
        \begin{equation}\label{LinearCaseEquations}
            \begin{gathered}
            p_{23} = \nu_2 \nu_4(\nu_1 \nu_2 \nu_3 \nu_4 - 256) t_1^2 t_3^2 + 16 \nu_4 (\nu_1 \nu_3 \nu_4 - 16 \nu_2) t_1^2+ 16 \nu_2 (\nu_1 \nu_2 \nu_3 - 16 \nu_4) t_3^2+256(\nu_1 \nu_3 - \nu_2 \nu_4),\\
            p_{24} = \nu_1 \nu_4^2(\nu_2^2 - 16) t_1^2 t_3 + 16\nu_2(\nu_4^2 - 16) t_1 t_3^2 + 16 \nu_2(\nu_4^2 - 16) t_1 + 16 \nu_1 (\nu_2^2 - 16) t_3.
            \end{gathered}    
        \end{equation}

\subsection{Cylindrical Algebraic Decomposition}

The most general method for algebraic analysis of a system of polynomial equations and inequalities over the real field is Cylindrical Algebraic Decomposition. For a given set of polynomials, it decomposes the solution space into a finite number of connected semialgebraic sets on which every polynomial has a constant sign or vanishes identically. The original idea was introduced by George Collins \cite{Collins1975} as an efficient computational method for quantifier elimination. A modern description with applicable algorithms can be found in \cite{England2020}.

\textbf{Example}. Let us illustrate this methodology to the system obtained in \cref{LinearCaseEquations} which corresponds to the linear dependence case \cref{eq27}. We use the following order of variables $[\nu_1, \nu_2, \nu_3, \nu_4, t_1, t_3]$ with additional inequality $[\nu_1 \nu_2 \nu_3 \nu_4 \ne 0]$. Then Cylindrical Algebraic Decomposition yields 136 cells (which are semialgebraic sets). For brevity, we describe only one of them,
\begin{equation}\label{eq32}
\nu_1 = \frac{256\nu_2 \nu_4(t_1^2 + 1)(t_3^2 + 1)}{\nu_3(\nu_4^2 t_1^2 + 16)(\nu_2^2 t_3^2 + 16)}, \quad \nu_3 = \frac{-16\nu_4(t_1^2 + 1)}{(\nu_4^2 - 16)t_1 t_3},    
\end{equation}
$$
t_1 > 0, \quad t_3 < 0, \quad \nu_4 > 4, \quad  \nu_2 > 4.
$$

\textcolor{black}{Together with \cref{eq27} it forms a 4-parameter family of solutions for flexible \textcolor{black}{skew Kokotsakis meshes}.}

    \subsection{Construction}\label{subsec5.3}
    
    \begin{figure}[ht]
        
        \centering
        \begin{overpic}[width=1.0\textwidth]{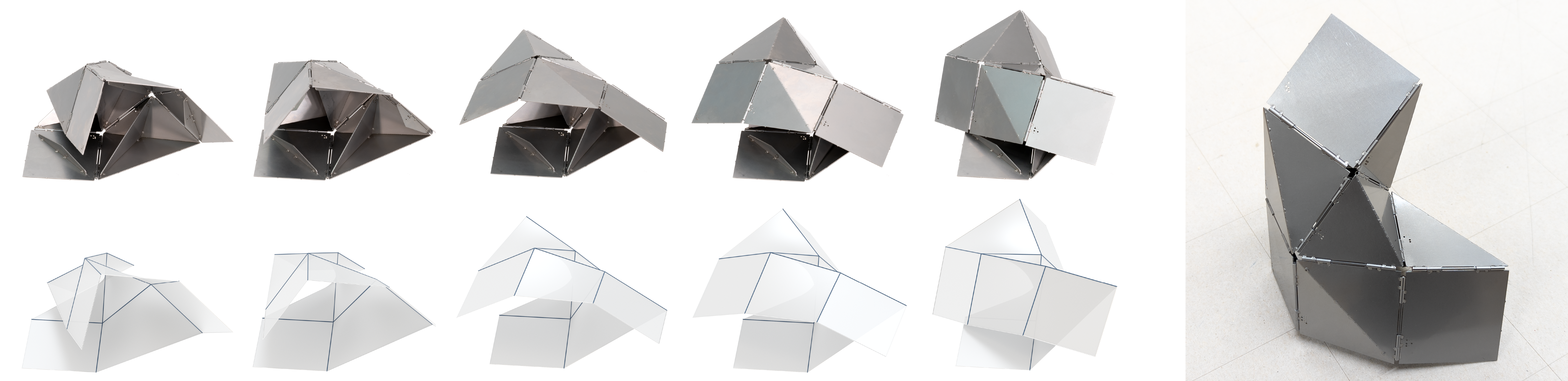}
        \put(89, 12.8){\contour{white}{$A_3$}}
        \put(88, 12){\Huge\textcolor{black}{.}}
        \put(84.8, 14.8){\contour{white}{$A_2$}}
        \put(82, 7.8){\Huge\textcolor{black}{.}}
        \put(80.75, 9.2){\contour{white}{$A_1$}}
        \put(85.7, 13.15){\Huge\textcolor{black}{.}}
        \put(90, 8.0){\contour{white}{$A_4$}}
        \put(89, 7.0){\Huge\textcolor{black}{.}}
        
        
        \end{overpic}

        \caption{Illustration of several configurations of our fabricated flexible Kokotsakis mesh. Photos and corresponding renderings are shown.} \label{fig:caigui1}
    \end{figure}
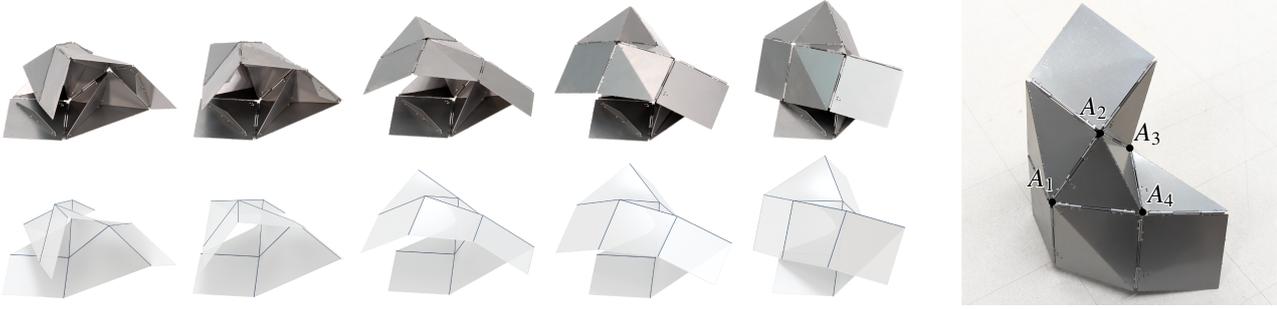   
    The example illustrated in Figure~\ref{fig:caigui1} is constructed\footnote{\textcolor{black}{The motion of the rendered examples has been obtained by numerically solving the Bricard equations.}} from the solution set \cref{eq32} \textcolor{black}{found in the previous section}, with the assumption that $t_i \neq 0$ and $\delta_i \neq \frac{\pi}{2}$. Firstly we observe that in the obtained cell, all $\nu_i$ are positive. We start 
    by choosing them as they correspond to planar angles in the central quad. Let us choose them as some positive integer numbers
     \begin{equation*}
        \nu_1 = 8, \quad \nu_2 = 16, \quad \nu_3 = 8, \quad \nu_4 = 16.
    \end{equation*}
    Formulas in \cref{eq32} immediately imply
    $$
    \begin{matrix}
        t_1 = \frac{\sqrt{14}}{14}, & t_2 = \frac{2\sqrt{14}}{7}, & t_3 = -\frac{\sqrt{14}}{14}, & t_4 = -\frac{2\sqrt{14}}{7}.
    \end{matrix}
    $$
    As $t_i \neq 0$, by \cref{elp-table} $\lambda_i = \mu_i = -1$ which is equivalent to
    \begin{equation*}
              \alpha_1 = \alpha_2 = \alpha_3 = \alpha_4  = \frac{\pi}{2}, \quad
            \gamma_1 = \gamma_2 = \gamma_3 = \gamma_4 = \frac{\pi}{2}.
    \end{equation*}
    Then, we can recover $\delta_i$ from $\nu_i$ as follows
    $$
    \begin{matrix}
        \delta_1 = \frac{\pi}{3}, & \delta_2 = \arccos\frac{1}{4}, & \delta_3 = \frac{\pi}{3}, & \delta_4 = \arccos\frac{1}{4}
    \end{matrix},
    $$
    and assuming the central skew face as two triangles with common edge $A_2 A_4$ between them, we can also recover $\tau_i$
    $$
    \begin{matrix}
        \tau_1 = \arccos\frac{1}{\sqrt{5}}, & \tau_2 = \arccos\frac{1}{\sqrt{5}}, & \tau_3 = \arccos\frac{1}{\sqrt{5}}, & \tau_4 = \arccos\frac{1}{\sqrt{5}}. 
    \end{matrix}
    $$
    Since $t_i = \tan(\tau_i + \zeta_i)$ we can find $\zeta_i$:
    $$
    \begin{matrix}
    \zeta_1 = \arctan\frac{\sqrt{14}}{14} - \arccos\frac{1}{\sqrt{5}},  & \zeta_2 = \arctan\frac{2\sqrt{14}}{7} - \arccos\frac{1}{\sqrt{5}},
    \end{matrix}
    $$
    $$
    \begin{matrix}
    \zeta_3 = \arctan(-\frac{\sqrt{14}}{14}) - \arccos\frac{1}{\sqrt{5}}, & \zeta_4 = \arctan(-\frac{2\sqrt{14}}{7}) - \arccos\frac{1}{\sqrt{5}}.
    \end{matrix}
    $$
    Finally, from the orthodiagonality condition, the last planar angles $\beta_i$ are
    \begin{equation*}
        \beta_1 = \frac{\pi}{2}, \quad \beta_2 = \frac{\pi}{2}, \quad \beta_3 = \frac{\pi}{2}, \quad \beta_4 = \frac{\pi}{2}.
    \end{equation*}

\begin{remark}
The central tetrahedron $A_1A_2A_3A_4$ of \cref{fig1} in this particular case is formed by two equilateral triangles perpendicular to each other.
\end{remark}

\subsection{Details on \textcolor{black}{the fabrication of the physical model}}
    \begin{figure}[ht]
        \centering
        {\includegraphics[width=1.0\textwidth]{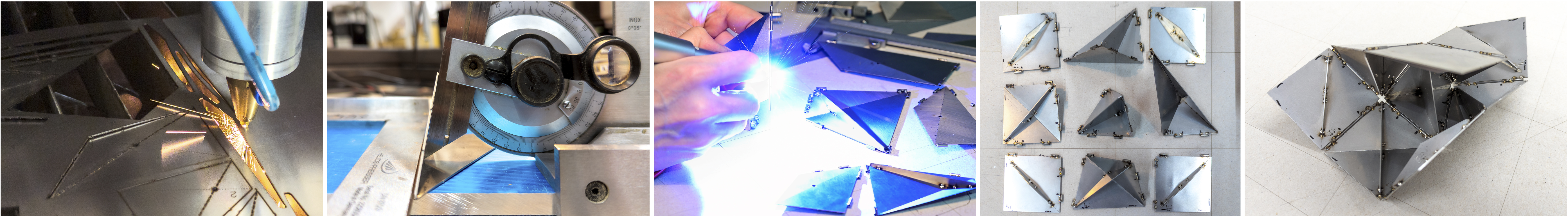}}
        \caption{\textcolor{black}{Illustration of the fabrication process of the flexible skew Kokotsakis mesh using stainless steel.} }\label{fig:florian1}
    \end{figure}
We demonstrate that the proposed mechanisms can be realized in practice without difficulties by building a small-scale prototype (ca. 180x180x180mm³). We built a precise model (see \cref{fig:florian1}) with very skew quads from the previous subsection.
Unlike Maleczek et al. \cite{maleczek2022rapid} and others, we decided against rapid prototyping and the use of plastics to avoid the large tolerances and inherent flexibility. We chose to use stainless steel to fabricate our model.

\textcolor{black}{The mechanism is built from a 0.8mm thick laser-cut hard-rolled stainless steel sheet (1.4404), CNC cut cold drawn stainless steel seamless capillary pipe (1.4301) with an outer diameter of 2.5mm and a wall thickness of 0.5mm, and straitened carbon spring steel wire (1.1200) with a diameter of 1.5mm. The laser-cut quads were manually bent until the desired angle was reached. V-grooves were ground into some quadrilaterals to facilitate large bending angles. Bracings and the tube cutoffs, forming the joint knuckles, were fusion pulse TIG welded to the bent quadrilaterals. The fit knuckle/pin in the joints is H10/h8 (knuckle inner diameter 1.5mm -0/+50µm, pin diameter 1.5mm +0/-14µm). In the longitudinal direction, the joint knuckles were pushed together with zero gap during welding. The resulting mechanism is stiff yet easy and smooth to move. It shows only one degree of motion and no backlash or wiggle. Its motion is solely a result of the geometry and not caused by any \textcolor{black}{material} flexibility, deformation, or loose fit.}
We have produced a video capturing the motion, which can be viewed on YouTube, \cite{kokotsakis23}.



\section{Conclusion}

\textcolor{black}{In this paper we made a first step in generalizing Izmestiev's method towards flexible skew Kokotsakis meshes. By doing so, we constructed the class of flexible skew Kokotsakis meshes of orthodiagonal involutive type. This remarkable class plays an important role in Izmestiev's classification. We were able to derive conditions for the skew case of OI type by studying the subdiagram similar to the planar case, though not explicitly mentioning it.}

Our belief is that all flexible skew Kokotsakis meshes may be classified similarly to Izmestiev's classification, potentially adding new subclasses to some families. However, as algebraic computations in the skew case become more complicated, we leave a complete classification for future work.

\section{Acknowledgements}

The authors are grateful to Caigui Jiang and Cheng Wang for their support with the graphical visualization of \textcolor{black}{flexible meshes}. \textcolor{black}{The valuable remarks of the anonymous reviewers are gratefully acknowledged.}. This work has been supported by KAUST baseline funding.

\bibliographystyle{unsrt}
\bibliography{main}

\appendix
\section{\textcolor{black}{Proof of the existence theorem}}
\label{apx:a}

\textcolor{black}{Set
\begin{equation*}
w_i=x_i+\frac{\mu_{i}}{x_i}, i=1,3; \quad w_j=x_j+\frac{\lambda_{j-1}}{x_j}, j=2,4; \quad t_n = \frac{2 F_n}{1 - F_n^2}, 1\leq n \leq 4;
\end{equation*}
$$
w_i'=y_i+\frac{\mu_{i-1}}{y_i}, i=1,3;\,\,\,\, w_j'=y_j+\frac{\lambda_{j}}{y_j}, j=2,4.
$$}
The flexibility of $(Q_1, Q_2, Q_3, Q_4)$ is equivalent to the system 
\begin{equation}\label{w&w'}
  w_1w_2=\nu_1, w_2'w_3'=\nu_2, w_3w_4=\nu_3, w_4'w_1'=\nu_4,
\end{equation}
where all $w_i,w_i'$ are the unknowns and $w_i'=k_i \frac{2 w_i + 4 t_i}{2 - t_i w_i}$ according to equation~\cref{eq28}. 
Belonging to the OI class implies that the system has infinitely many complex solutions. By solving the system, it is not hard to see that we can take any $w_i$ or $w_i'$ as a parameter and the others can be expressed rationally in this parameter. As a consequence, the system always has real solutions in $w_i, w_i'$. \textcolor{black}{We can recover real values for $x_i$ and $y_i$ if
\begin{equation}\label{cond_for_mu}
    |w_i| > 2\sqrt{\mu_i}, \quad |w'_i| > 2\sqrt{\mu_{i - 1}}, i = 1, 3; \quad |w_j| > 2\sqrt{\lambda_{j-1}}, \quad |w'_j| > 2\sqrt{\lambda_{j}}, i = 2, 4.
\end{equation}}
The global restriction might be required to make sure that system (\ref{w&w'}) admits real solutions in all $x_i,y_i$. In fact, since the relation between $(w_i,w_i')$ is actually induced by equation (\ref{eq8}), it is sufficient to only consider the real solutions for all $x_i$.

We first partition the involution factors into cyclically arranged sets $L_i$ as follows,
$$
L_1:=\{\mu_4,\mu_1\}, L_2:= \{\lambda_1,\lambda_2\}, L_3:=\{\mu_2,\mu_3\}, L_4:=\{\lambda_3,\lambda_4\}.
$$
According to table~\cref{elp-table}, the elements in each partition $L_i$ share the same sign $s_i$ of factors. 

Case 1: All $s_i$ are -1. This case is trivial, since we can recover real values of $x_i$ from any real value of $w_i$. In
this case, we always have a real solution.

Case 2: Only one  $s_i$ is $+1$. After relabeling we may assume $s_1=1$. Since there are no restrictions for $w_i$ ($i\neq 1$), we can just take $w_1$ as a parameter to solve the system and keep it away from small numbers $i$ order to get real solutions for~$x_1$ and hence for all $x_i$.

Case 3: Two adjacent $s_i$ are $+1$. After relabeling we may assume $s_1=s_2=1$. The existence of real solutions for all $x_i$ is simply reduced to the local restriction on $Q_1$, i.e. $\frac{\nu_1^2}{\lambda_1 \mu_1} > 16$.

Case 4: Two non-adjacent $s_i$ are $+1$. After relabeling we may assume $s_1=s_3=1$. There are no restrictions on $w_2$ and $w_4$ given that $s_2=s_4=-1$. \textcolor{black}{This is exactly the case in which $\lambda_i\mu_i<0$ for all $i$.}

\textcolor{black}{Suppose that $\mu_i > 0$.} Then according to equation~\cref{eq15} the problem is reduced to whether 
$$
k_2\frac{4t_2w_1+2\nu_1}{2w_1-\nu_1t_2}w_3'=\nu_2
$$
\textcolor{black}{admits real solutions in which both $|w_1|$ and $|w_3'|$ satisfy the conditions \cref{cond_for_mu}. More specifically, the following system of inequalities must have solutions (as intervals)
\begin{equation*}
    |w_1|> 2\sqrt{\mu_1}, \quad |w_3'|=\left|\frac{\nu_2(2w_1-\nu_1t_2)}{k_2(4t_2w_1+2\nu_1)}\right|> 2\sqrt{\mu_2}
\quad \text{or equivalently} \quad
|w_1|> 2\sqrt{\mu_1}, \quad \left|\frac{2w_1-\nu_1t_2}{4t_2w_1+2\nu_1}\right|> 2\sqrt{\mu_2}\left|\frac{k_2}{\nu_2}\right|.
\end{equation*}}
Clearly, the system has solutions when $t_2=0$. So we may assume $t_2 \ne 0$, i.e. $F_2\notin \{0,\pm\infty\}$ and hence $k_2=1$ by Lemma~\ref{lma3.4}.
Then we have
\begin{equation}\label{ineqs}
|w_1|> 2\sqrt{\mu_1}, \quad \left|\frac{2w_1-\nu_1t_2}{2t_2w_1+\nu_1}\right|> 4 \frac{\sqrt{\mu_2}}{|\nu_2|}.
\end{equation}
Let us  assume that (\ref{ineqs}) has no
solution. If we regard $w_1$ as a variable in $\mathbb{R}\text{P}^1\cong S^1$, then any interval such as $(a,b)$ or $(-\infty,a)\cup(b,+\infty)$ is just an arc of $S^1$ (or $\mathbb{R}\text{P}^1$). Set
$$
I:=\{w_1\in \mathbb{R}\text{P}^1:|w|> 2\sqrt{\mu_1}\}, \quad J:=\{w_1\in \mathbb{R}\text{P}^1:|w_1|> 4\frac{\sqrt{\mu_2}}{|\nu_2|}\}.
$$
\begin{figure}
\hfill
\begin{overpic}[width=0.5\textwidth]{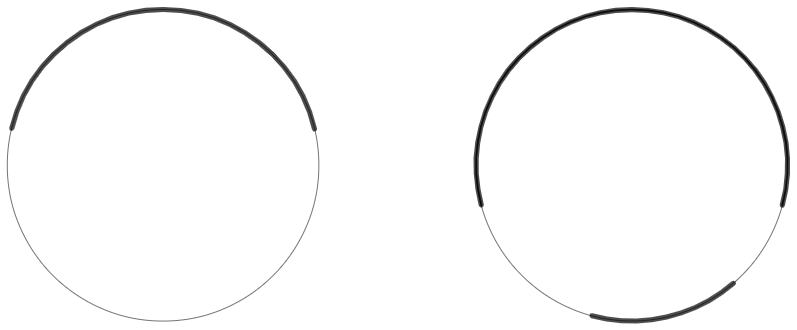}
    \put(24,37){\footnotesize\contour{white}{$I$}}
    \put(75,37){\footnotesize\contour{white}{$J$}}    
    \put(76,10){\footnotesize\contour{white}{$f(I)$}}
\end{overpic}
  \hfill{}
    \caption{Illustration of case 4 in the proof of Theorem~\ref{thm:main}.}
    \label{RP1}
\end{figure}
The M\"obius transformation
$$
f:\mathbb{R}\text{P}^1\rightarrow \mathbb{R}\text{P}^1, \quad f(w_1)=\frac{2w_1-\nu_1t_2}{2t_2w_1+\nu_1},
$$
maps an interval to an interval. According to our assumption, $f(I), J$ must be disjoint, see Figure ~\ref{RP1}. The end points of interval $f(I)$ are just $f(\pm 2\sqrt{\mu_1})$ given that $f$ is one to one. So by properly choosing from $w_1=\pm 2\sqrt{\mu_1}$ we can obtain 
$$
|f(w_1)| \leq \max(|f(\pm 2\sqrt{\mu_1})|) = \frac{4\sqrt{\mu_1}+|\nu_1t_2|}{|4|t_2|\sqrt{\mu_1}-|\nu_1||}\leq 4\sqrt{\mu_2}/|\nu_2|.
$$
On the other hand, $\infty \notin f(I)$ since $\infty \in J$. So we have $\frac{-\nu_1}{2t_2}=f^{-1}(\infty) \notin I$, i.e.
$$
4 |t_2|\sqrt{\mu_1}-|\nu_1| \geq 0.
$$
Since $4\sqrt{\mu_1}+|\nu_1t_2|, 4\sqrt{\mu_2}/|\nu_2|$ are always positive, the above two inequalities can be combined into a single one 
$$
4\sqrt{\mu_1}+|\nu_1t_2|\leq 4\sqrt{\mu_2}(4|t_2|\sqrt{\mu_1}-|\nu_1|)/|\nu_2|.
$$
Note that this inequality actually characterizes $f(I)\cap J=\varnothing$. This is because $f(\pm 2\sqrt{\mu_1})$ separate $\mathbb{R}\text{P}^1$ into two intervals, only one of which contains $\infty$. So $f(I)\cap J=\varnothing$ as long as $f(\pm 2\sqrt{\mu_1})\notin J$ and $\infty \notin f(I)$. In other words, system (\ref{ineqs}) has continuous solutions if and only if
$$
4\sqrt{\mu_1}+|\nu_1t_2|> 4\sqrt{\mu_2}(4|t_2|\sqrt{\mu_1}-|\nu_1|)/|\nu_2|,
$$
or equivalently,
$$
16\sqrt{\mu_1\mu_2}-|\nu_1\nu_2| < \frac{4|\nu_2|\sqrt{\mu_1}+4|\nu_1|\sqrt{\mu_2}}{|t_2|}.
$$

\textcolor{black}{In the case $\lambda_i > 0$ we consider the new linkage $(\tilde{Q_1}, \tilde{Q_2}, \tilde{Q_3}, \tilde{Q_4})$ obtained from $(Q_1, Q_2, Q_3, Q_4)$ by cyclic rotation such that $\tilde{Q}_{i} = Q_{i + 1}$. Note that under such a rotation we have the relabeling $\tilde{\alpha}_i = \gamma_{i + 1}$, $\tilde{\gamma}_{i} = \alpha_{i+1}$ and $\tilde{t}_{i} = t_{i + 1}$ which leads to the inequalities $\tilde{\mu}_i = \lambda_{i + 1} > 0$ and $\tilde{\lambda_i} = \mu_{i + 1} < 0$. Therefore, the condition for the existence of the real solutions for the new complex is \cref{global-existence-condition}. After returning to the initial labeling we have the desired inequality \cref{global-existence-condition1}.} It is the only case which requires a global restriction.

        Case 5: More than two $s_i$ equal +1. After relabeling we may assume $s_1=s_2=s_3=1$ and then we have $t_i = 0$ for all indices because of \cref{eq18}. It implies $F_i\in\{0, \infty\}$ for all $i$. In any of these cases, we have $w_i'=k_iw_i$
        where $k_i$ is constant. The local restriction $\frac{\nu_i^2}{\lambda_i \mu_i} > 16$ must be applied for all $i$.  We take $w_1$ as a parameter. Then in order to recover real values for $x_1$, the domain of $|w_1|$ is the interval $(2\sqrt{\mu_1}, +\infty)$.  The first two equations in \cref{w&w'} have the form 
        \begin{equation*}
            w_1 w_2 = \nu_1, \quad w_2' w_3' = k_2 k_3 w_2 w_3 = \nu_2,
        \end{equation*} and thus the domains of $|w_2|, |w_3|$ in order to recover real values for $x_2, x_3$ are 
        \begin{equation*}
            |w_2| \in \Big(2\sqrt{\lambda_1}, \frac{|\nu_1|}{2\sqrt{\mu_1}}\Big), \quad |w_3| \in \Big(\frac{2\sqrt{\mu_1}|\nu_2|}{|k_2 k_3\nu_1|}, \frac{|\nu_2|}{2|k_2 k_3|\sqrt{\lambda_1}}\Big) \cap (2\sqrt{\mu_3}, +\infty).
        \end{equation*}
        The last intersection is not empty since $|\nu_2| > 4|k_2 k_3|\sqrt{\lambda_1\mu_3}$ is equivalent to $|\nu_2| > 4\sqrt{\lambda_2 \mu_2}$ because of~\cref{eq10} as $|k_2 k_3|\sqrt{\lambda_1 \mu_3} = \sqrt{\lambda_2 \mu_2}$ for every $k_2, k_3$.
        
        Similarly, for the last two equations in (\ref{w&w'}), we start with $|w_3|$ as parameter and obtain the following domains
        \begin{equation*}
            |w_3| \in (2\sqrt{\mu_3}, +\infty), \quad |w_4| \in \Big(2\sqrt{\lambda_3}, \frac{\nu_3}{2\sqrt{\mu_3}}\Big), \quad |w_1| \in \Big(\frac{2\sqrt{\mu_3}|\nu_4|}{|k_1 k_4\nu_3|}, \frac{|\nu_4|}{2|k_1 k_4|\sqrt{\lambda_3}}\Big) \cap (2\sqrt{\mu_1}, +\infty),
        \end{equation*}~where the last intersection is also not empty.


\section{Deltoids and antideltoids}
\label{sec:appendix}

\subsection{OI type with two (anti)deltoids}
\begin{enumerate}
    \item Planar angles $(\alpha_i , \beta_i, \gamma_i, \delta_i)$, where $i = 1, 4$, satisfy the conditions:
    $$
    \begin{matrix}
    \cos(\alpha_i) \cos(\gamma_i) = \cos(\beta_i) \cos(\delta_i) & \alpha_i \pm \beta_i \pm \gamma_i \pm \delta_i \neq 0 \text{ mod } 2\pi
    \end{matrix}
    $$
    and planar angles $(\alpha_i, \beta_i, \gamma_i, \delta_i)$, where $i =2, 3$, both satisfy one of the conditions:
    \begin{equation*}
            \alpha_i = \beta_i, \quad \gamma_i = \delta_i \quad (\text{or} \quad \alpha_i = \pi - \beta_i, \quad \gamma_i = \pi - \delta_i);   
    \end{equation*}
    \item The couplings of adjacent quadrilaterals are compatible, see~\cref{def3.6};
    \item The set of parameters $(\nu_1, \xi_2, \xi_3, \nu_4, t_1, t_2, F_3, t_4)$ satisfies the following system(in the case of antideltoids the system coincides with the one below after changing $F_3$ to $-F_3$)
    \begin{equation*}\label{eq24}
        \begin{gathered}
            p_{12} = -2 \xi_3 \nu_4 t_1 t_2 t_4 + 2 \xi_2 \nu_4 t_1 + 8 \xi_3 t_2 + 8 \xi_2 t_4 + F_3(4 \nu_4 t_1 t_2 + \xi_2 \xi_3 \nu_4 t_1 t_4 + 16 t_2 t_4 - 4 \xi_2 \xi_3),\\
            p_{13} = \xi_3(-\nu_1 \xi_2 \nu_4 t_1 t_2 t_4 + 16  \xi_2 t_1 + 4 \nu_1 \xi_2 t_2 + 4 \xi_2 \nu_4 t_4 + F_3(32 t_1 t_2 + 2 \nu_1 \nu_4 t_1 t_4 + 8 \nu_4 t_2 t_4 - 8 \nu_1)),\\
            p_{14} =  2(16 \xi_2 - \nu_1 \xi_3 \nu_4) t_1 t_4 + 8(\nu_1 \xi_3 - \xi_2 \nu_4) + F_3((64 - \nu_1 \xi_2 \xi_3 \nu_4) t_1 t_2 t_4 + 4( \nu_1 \xi_2 \xi_3 - 4 \nu_4) t_2),\\
            p_{23} = (\nu_1 \xi_2 \nu_4 - 16 \xi_3) t_1 t_2 + 4 (\nu_1 \xi_2 - \xi_3 \nu_4) t_2 t_4 + F_3(2 (4 \xi_2 \xi_3 - \nu_1 \nu_4) t_1 + 2 (\xi_2 \xi_3 \nu_4 - 4 \nu_1) t_4),\\
            p_{24} = -64 t_1 t_2 t_4  + 4 \nu_1 \nu_4 t_1 + 16 \nu_4 t_2 + 16 \nu_1 t_4 + F_3(2 \nu_1 \xi_2 \nu_4 t_1 t_2 + 32 \xi_2 t_1 t_4 + 8 \nu_1 \xi_2 t_2 t_4 - 8 \xi_2 \nu_4),\\
            p_{34} = \nu_1 (-8 \xi_2 t_1 t_2 t_4 + 8\xi_3 t_1 + 2 \xi_2 \nu_4 t_2 + 2 \xi_3 \nu_4 t_4 + F_3(4 \xi_2 \xi_3 t_1 t_2 + 16 t_1 t_4 + \xi_2 \xi_3 \nu_4 t_2 t_4 - 4 \nu_4)).
        \end{gathered}    
    \end{equation*}
\end{enumerate}
\subsection{OI type with deltoid and antideltoid}\label{secA2}
This subclass has no counterpart in the planar case.  
\begin{enumerate}
     \item Planar angles $(\alpha_i , \beta_i, \gamma_i, \delta_i)$, $i = 1, 4$, satisfy the conditions:
    $$
    \begin{matrix}
    \cos(\alpha_i) \cos(\gamma_i) = \cos(\beta_i) \cos(\delta_i) & \alpha_i \pm \beta_i \pm \gamma_i \pm \delta_i \neq 0 \text{ mod } 2\pi
    \end{matrix}
    $$
    and planar angles $(\alpha_i, \beta_i, \gamma_i, \delta_i)$, where $i = 2, 3$, satisfy the conditions:
    \begin{equation*}
        \begin{matrix}
            \alpha_2 = \beta_2\text{, } & \gamma_2 = \delta_2\text{, } & \alpha_3 = \pi - \beta_3\text{, } & \gamma_3 = \pi - \delta_3
        \end{matrix}
    \end{equation*}
    \item The couplings of adjacent quadrilaterals are compatible, see~\cref{def3.6};
    \item\label{eq25} The set of parameters $(\nu_1, \xi_2, \xi_3, \nu_4, t_1, t_2, F_3, t_4)$ satisfies the system
    \begin{equation*}
        \begin{gathered}
            p_{12} = 2 \xi_3 \nu_4 t_1 t_2 F_3 t_4 - 4 \nu_4 t_1 t_2 + 2 \xi_2 \nu_4 t_1 F_3 + \xi_2 \xi_3 \nu_4 t_1 t_4 - 8 \xi_3 t_2 F_3 - 16 t_2 t_4 + 8 \xi_2 F_3 t_4 - 4 \xi_2 \xi_3,\\
            p_{13} = 32 t_1 t_2 F_3 t_4 + \nu_1 \xi_2 \nu_4 t_1 t_2 - 2 \nu_1 \nu_4 t_1 F_3 + 16 \xi_2 t_1 t_4 - 8 \nu_4 t_2 F_3 + 4 \nu_1 \xi_2 t_2 t_4 - 8 \nu_1 F_3 t_4 - 4 \xi_2 \nu_4,\\
            p_{14} = (\nu_1 \xi_2 \nu_4 + 16 \xi_3) t_1 t_2 F_3 + 4(\nu_1 \xi_2 + \xi_3 \nu_4) t_2 F_3 t_4 + 2(\nu_1 \nu_4 + 4 \xi_2 \xi_3) t_1 + 2(4\nu_1 + \xi_2 \xi_3 \nu_4) t_4,\\
            p_{23} = (\nu_1 \xi_2 \xi_3 \nu_4 + 64) t_1 t_2 t_4 - 2(16 \xi_2 + \nu_1 \xi_3 \nu_4) t_1 F_3 t_4 - 4(4 \nu_4 + \nu_1 \xi_2 \xi_3)t_2 + 8 (\xi_2 \nu_4 + \nu_1 \xi_3) F_3,\\
            p_{24} = \xi_3 (\nu_1 \xi_2 \nu_4 t_1 t_2 F_3 t_4 + 32 t_1 t_2 - 16 \xi_2 t_1 F_3 + 2 \nu_1 \nu_4 t_1 t_4 - 4 \nu_1 \xi_2 t_2 F_3 +  8 \nu_4 t_2 t_4 - 4 \xi_2 \nu_4 F_3 t_4 - 8 \nu_1),\\
            p_{34} = \nu_1 (8 \xi_2 t_1 t_2 F_3 t_4 - 4 \xi_2 \xi_3 t_1 t_2 + 8 \xi_3 t_1 F_3 + 16 t_1 t_4 - 2 \xi_2 \nu_4 t_2 F_3 - \xi_2 \xi_3 \nu_4 t_2 t_4 + 2 \xi_3 \nu_4 F_3 t_4 - 4 \nu_4).
        \end{gathered}
    \end{equation*}    
\end{enumerate}
\textbf{Example.}
Here we present the dihedral and planar angles for a flexible Kokotsakis mesh of OI type with deltoid and antideltoid. The example is constructed under the assumption that $t_1 t_2 t_4 \neq 0$, $F_3 \neq 0$ with the following angles
\begin{figure}[ht!]
        \centering
        \includegraphics[width=0.85\textwidth]{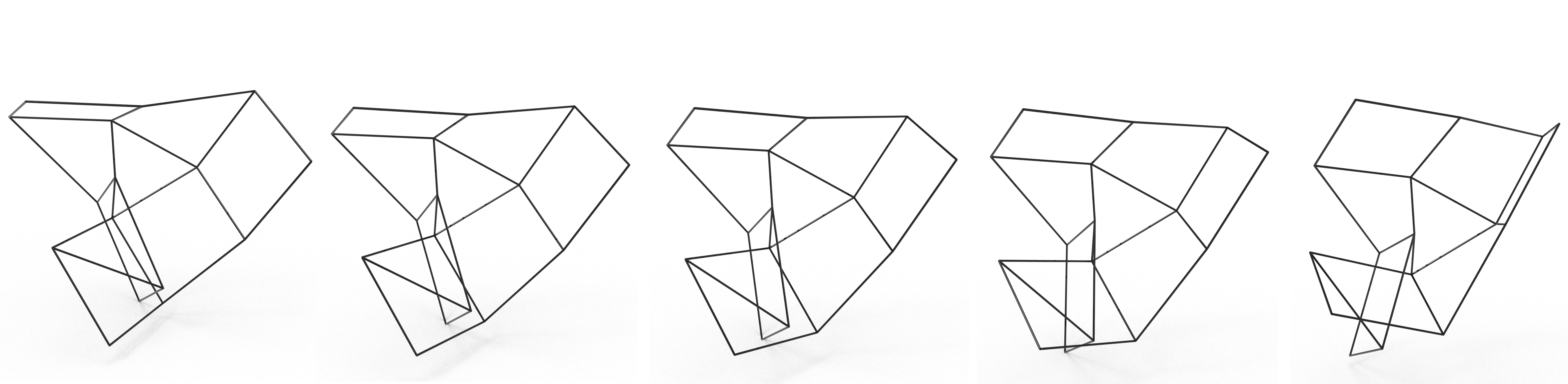}
        \caption{Some positions of a deltoid-antideltoid \textcolor{black}{mesh}. Due to the arising self-intersections we use a wireframe visualization.}
\end{figure}
\begin{equation*}
    \begin{matrix}
        \delta_1 = \frac{\pi}{3}, & \delta_2 = \arccos\frac{\sqrt{3}}{4}, & \delta_3 = \frac{2\pi}{3}, & \delta_4 = \arccos\frac{\sqrt{3}}{4},\\
    \end{matrix}
\end{equation*}
\begin{equation*}
    \begin{matrix}
        \tau_1 = \arccos\frac{3}{\sqrt{13}},& \tau_2 = \arccos\frac{3}{\sqrt{13}},& \tau_3 = \arccos\frac{1}{13},& \tau_4 = \arccos\frac{1}{13},
    \end{matrix}
\end{equation*}
\begin{equation*}
    \begin{matrix}
        t_1 = 10,& t_2 = \frac{-3597+2600\sqrt{3}}{5(39 + 12218\sqrt{3})}, & F_3 = \frac{1}{10}, & t_4 = \frac{-3597+2600\sqrt{3}}{-5(39 + 12218\sqrt{3})},
    \end{matrix}
\end{equation*}
\begin{equation*}
    \begin{matrix}
        \zeta_1 = \arctan t_1 - \tau_1,& \zeta_2 = \arctan t_2 - \tau_2,& \zeta_3 = 2\arctan F_3 - \tau_3,& \zeta_4 = \arctan t_4 - \tau_4,\\
    \end{matrix}
\end{equation*}
\begin{equation*}
    \alpha_1 = \alpha_2 = \alpha_3 = \alpha_4 = \frac{\pi}{2}, \quad \gamma_1 = \gamma_2 = \gamma_3 = \gamma_4 = \frac{\pi}{2}, \quad \beta_1 = \beta_2 = \beta_3 = \beta_4 = \frac{\pi}{2}.
\end{equation*}

\end{document}